\documentclass[12pt, a4paper, reqno]{amsart}

\usepackage{amsmath}
\usepackage{amsthm, mathtools, amssymb, latexsym, mathrsfs,
todonotes, eufrak, color}
\usepackage{graphicx}
\usepackage[english]{babel}
\usepackage[T1]{fontenc}
\usepackage{diagbox}
\usepackage{lmodern}

\usepackage[all]{xy}
\usepackage{nicefrac}
\usepackage{mparhack}
\usepackage[normalem]{ulem}
\usepackage{color}
\usepackage{blindtext}
\usepackage{fancyhdr}
\usepackage{enumerate}
\usepackage{hyperref}
\usepackage{url}

\DeclareFontFamily{OT1}{pzc}{}
\DeclareFontShape{OT1}{pzc}{m}{it}{<-> s * [1.10] pzcmi7t}{}
\DeclareMathAlphabet{\mathpzc}{OT1}{pzc}{m}{it}

\newtheorem{dummy}{Dummy}
\numberwithin{dummy}{section}
\numberwithin{figure}{section}
\numberwithin{table}{section}

\theoremstyle{plain}
\newtheorem{theorem}[dummy]{Theorem}
\newtheorem{theo}[dummy]{Theorem}

\newtheorem{lemma}[dummy]{Lemma}
\newtheorem{proposition}[dummy]{Proposition}
\newtheorem{prop}[dummy]{Proposition}

\newtheorem{corollary}[dummy]{Corollary}

\theoremstyle{definition}
\newtheorem{definition}[dummy]{Definition}

\theoremstyle{remark}
\newtheorem{remark}[dummy]{Remark}

\makeatletter
\def\imod#1{\allowbreak\mkern10mu({\operator@font mod}\,\,#1)}
\makeatother

\numberwithin{equation}{section}

\newcommand{\mmid}[0]{\parallel}
\newcommand{\ppar}[0]{\ \par}

\newcommand{\Z}{\mathbb{Z}}

\newcommand{\N}{\mathbb{N}}

\newcommand{\cC}{\mathcal{C}}

\newcommand{\s}{{e_s}}
\newcommand{\T}{{f_s}}

\renewcommand{\epsilon}{\varepsilon}
\renewcommand{\theta}[0]{\vartheta}

\newcommand{\Gal}{\mathop{\mathrm{Gal}}\nolimits}

\newcommand{\ord}{\mathop{\mathrm{ord}}\nolimits}

\newcommand{\Int}{{\rm Inn}}

\newcommand{\Size}[1]{\left\lvert #1 \right\rvert}
\newcommand{\Set}[1]{\left\{\, #1 \,\right\}}

\newcommand{\Span}[1]{\left\langle\, #1 \,\right\rangle}

\newcommand{\Perm}{S}

\DeclareMathOperator{\End}{End}
\DeclareMathOperator{\Aut}{Aut}
\DeclareMathOperator{\Inn}{Inn}

\DeclareMathOperator{\Hol}{Hol}

\DeclareMathOperator{\inv}{inv}

\newcommand{\norm}[0]{\trianglelefteq}

\newcounter{enumi_saved}

\makeatletter
\def\imod#1{\allowbreak\mkern10mu({\operator@font mod}\,\,#1)}
\makeatother

\allowdisplaybreaks

\begin{document}

\date{21 aprile 2020 8:14 CEST --- Version 5.09%
}

\title[Hopf-Galois structures and skew braces]%
      {Hopf-Galois structures on \\ extensions of degree $p^{2} q$\\
        and skew braces of order $p^{2} q$:\\
        the cyclic Sylow $p$-subgroup case}
      
\author{E. Campedel}

\address[E.~Campedel]%
{Dipartimento di Matematica e Applicazioni\\
Edificio U5\\
Universit\`a degli Studi di Milano-Bicocca\\
via Roberto Cozzi, 55\\
20126 Milano}

\email{e.campedel1@campus.unimib.it}

\author{A. Caranti}

\address[A.~Caranti]%
 {Dipartimento di Matematica\\
  Universit\`a degli Studi di Trento\\
  via Sommarive 14\\
  I-38123 Trento\\
  Italy} 

\email{andrea.caranti@unitn.it} 

\urladdr{http://science.unitn.it/$\sim$caranti/}

\author{I. Del Corso}

\address[I.~Del Corso]%
        {Dipartimento di Matematica\\
          Universit\`a di Pisa\\
          Largo Bruno Pontecorvo, 5\\
          56127 Pisa\\
          Italy}
\email{ilaria.delcorso@unipi.it}

\urladdr{http://people.dm.unipi.it/delcorso/}

\subjclass[2010]{12F10 16W30 20B35 20D45}

\keywords{Hopf-Galois extensions, Hopf-Galois structures, holomorph,
  regular subgroups, braces, skew braces}

\begin{abstract}    
  Let $p, q$ be distinct primes, with $p > 2$.
  
  We  classify  the Hopf-Galois  structures  on  Galois extensions  of
  degree $p^{2}  q$, such that  the Sylow $p$-subgroups of  the Galois
  group are cyclic.
  
  This  we  do, according  to  Greither  and  Pareigis, and  Byott,  by
  classifying the  regular subgroups  of the  holomorphs of  the groups
  $(G,  \cdot)$  of  order  $p^{2}  q$, in  the  case  when  the  Sylow
  $p$-subgroups of  $G$ are cyclic.  This is equivalent  to classifying
  the skew braces $(G, \cdot, \circ)$.
  
  Furthermore, we  prove that if $G$  and $\Gamma$ are groups  of order
  $p^{2} q$ with non-isomorphic Sylow  $p$-subgroups, then there are no
  regular subgroups  of the  holomorph of $G$  which are  isomorphic to
  $\Gamma$. Equivalently, a Galois extension with Galois group $\Gamma$
  has no Hopf-Galois structures of type $G$.
  
  Our method relies  on the alternate brace operation  $\circ$ on $G$,
  which we use  mainly indirectly, that is, in terms  of the functions
  $\gamma : G  \to \Aut(G)$ defined by $g \mapsto  (x \mapsto (x \circ
  g) \cdot g^{-1})$.  These functions are in one-to-one correspondence
  with  the  regular  subgroups  of  the holomorph  of  $G$,  and  are
  characterised by the functional equation $\gamma(g^{\gamma(h)} \cdot
  h) = \gamma(g) \gamma(h)$, for $g,  h \in G$.  We develop methods to
  deal with these functions, with  the aim of making their enumeration
  easier,  and  more conceptual.
\end{abstract}

\thanks{The   first   and   the   second   author   are   members   of
  INdAM---GNSAGA. The authors gratefully  acknowledge support from the
  Departments of  Mathematics of  the Universities  of Milano-Bicocca,
  Pisa, and Trento.  {The third  author has performed this activity in
    the framework of  the PRIN 2017, title  ``Geometric, algebraic and
    analytic methods in arithmetic''.}}

\maketitle

\thispagestyle{empty}

\section{Introduction}
\subsection{The general problem, and the  classical approach}

Let   $L/K$   be   a   finite  Galois   field   extension,   and   let
$\Gamma=\Gal(L/K)$.  Then the group  algebra $K[\Gamma]$ is a $K$-Hopf
algebra, and its natural action on $L$ endows $L/K$ with a Hopf-Galois
structure. In  general this is  not the only Hopf-Galois  structure on
$L/K$, and  the study of  Galois module structures different  from the
classical one  is important, for  example in the context  of algebraic
number theory. In  fact, when $L/K$ is a wildly  ramified extension of
local  fields,  there  are  cases   in  which  the  ring  of  integers
$\mathcal{O}_{L}$ of  $L$ is free  as a module  over a Hopf  order in
some  $K$-Hopf  algebra  $H$,  but not  in  $K[\Gamma]$  (see  Child's
book~\cite{ChildsBook} for an overview and for the specific results).

Greither and Pareigis~\cite{GP} showed that the Hopf-Galois structures
on  $L/K$  correspond  to  the  regular subgroups  $G$  of  the  group
$\Perm(\Gamma)$  of  permutations  on  the  set  $\Gamma$,  which  are
normalised  by   the  image   $\rho(\Gamma)$  of  the   right  regular
representation $\rho$  of $\Gamma$ (in  the relevant literature  it is
common to use the left regular representation $\lambda$ instead of the
right  one $\rho$  we  are  employing here.   We  have translated  the
statements in the literature from left to right).

The groups  $G$ and $\Gamma$ have  the same cardinality but  they need
not be  isomorphic.  We will  say that  a Hopf-Galois structure  is of
\emph{type}  $G$  if  $G$  is  the  group  associated  to  it  in  the
Greither-Pareigis correspondence.

As usual we  denote by $e(\Gamma, G)$ the number  of regular subgroups
of $\Perm(\Gamma)$ normalised by  $\rho(\Gamma)$, which are isomorphic
to $G$.   Equivalently, $e(\Gamma,  G)$ is  the number  of Hopf-Galois
structures of type  $G$ on a Galois field extension  with Galois group
isomorphic to $\Gamma$.

The direct  determination of all regular  subgroups of $\Perm(\Gamma)$
normalised by $\rho(\Gamma)$ is in general a difficult task, since the
groups  $\Perm(\Gamma)$ is  large.   However, Childs~\cite{Chi89}  and
Byott~\cite{Byo96}  observed that  the  condition that  $\rho(\Gamma)$
normalises  $G$  can  be  reformulated  by  saying  that  $\Gamma$  is
contained in the holomorph $\Hol(G)$ of $G$, regarded as a subgroup of
$\Perm(G)$.   This translation  turns  out to  be  very useful,  since
$\Hol(G)$ is  usually much  smaller than $\Perm(\Gamma)$.  One obtains
the following result.
\begin{theorem}[{{\cite[Corollary p.~3320]{Byo96}}}]
  \label{th:byott96}
  Let  $L/K$ be  a finite  Galois  field extension  with Galois  group
  $\Gamma$.  For  any group $G$  with $\Size{G} =  \Size{\Gamma}$, let
  $e'(\Gamma,G)$  be  the number  of  regular  subgroups of  $\Hol(G)$
  isomorphic to $\Gamma$.

  Then the number $e(\Gamma, G)$ of Hopf-Galois structures on $L/K$ of
  type $G$ is given by
  \begin{equation}
    \label{eq:e_vs_e-prime}
    e(\Gamma,   G)  =   \frac{\Size{\Aut(\Gamma)}}{\Size{\Aut(G)}}  \,
    e'(\Gamma, G).
  \end{equation}

  Moreover $e(\Gamma)$, the total  number of Hopf-Galois structures on
  $L/K$, is given by $\sum_{G}e(\Gamma, G)$  where the sum is over all
  isomorphism types $G$ of groups of order $\Size{\Gamma}$.
\end{theorem}

There  is  a rich  literature  on  Hopf-Galois structures  on  various
classes of field extensions (for a  survey up to 2000, see the already
mentioned book by Childs~\cite{ChildsBook}).   For example, for an odd
prime $p$ it is known  that there are $p^{m-1}$ Hopf-Galois structures
on a  cyclic extension  of degree  $p^m$, and they  are all  of cyclic
type~\cite{kohl98}.  An elementary abelian  extension of degree $p^m$,
with $p>m$, has at least $p^{m(m-1)-1}(p-1)$ Hopf-Galois structures of
abelian type (see~\cite{Childs05}).  An abelian, non cyclic, extension
of degree $p^m$  admits also structures of non-abelian  type for $m\ge
3$~\cite{ByoChi}.

The exact number of Hopf-Galois  structures is known for some families
of Galois  extensions, such as those  of order the square  of a prime,
\cite{Byo96}, of  order the  product of two  primes~\cite{Byott04}, of
order the cube of a prime~\cite{zen2018, zen2019} and for the cyclic extensions
of squarefree order~\cite{BA17}.

An interesting issue is also  to determine which general properties of
either $\Gamma$  or $G$ force those  of the other. For  certain Galois
groups $\Gamma$ it is known that every Hopf-Galois structure must have
type $\Gamma$ (see~\cite{ByoChi}). For a Galois extension whose Galois
group $\Gamma$ is  abelian, the type $G$ of  any Hopf-Galois structure
must  be soluble~\cite[Theorem  2]{Byott15}, although  for a  soluble,
non-abelian Galois group $\Gamma$  there can be Hopf-Galois structures
whose   type   is   not   soluble~\cite[Corollary   3]{Byott15};   see
also~\cite{Tsang2019}.

The study of Hopf-Galois structures,  that is, of regular subgroups of
holomorphs, has also  a deep connection with the  theory of \emph{skew
  braces}. In  fact, if $G$ is  a group with respect  to the operation
``$\cdot$'',  classifying  the  regular   subgroups  of  $\Hol(G)$  is
equivalent to determining the operations  ``$\circ$'' on $G$ such that
$(G, \cdot, \circ)$ is a (right) skew brace~\cite{skew}, that is, $(G,
\circ)$ is also a  group, and the two group structures  on the set $G$
are related by
\begin{equation}
  \label{eq:brace}
  (g \cdot h) \circ k = (g \circ k) \cdot k^{-1} \cdot (h \circ k).
\end{equation}
This connection was first  observed by Bachiller in~\cite[\S 2]{Bac16}
and it is described in detail in the appendix to~\cite{SV2018}.

\subsection{The methods}

The main goal of this paper  is to classify the Hopf-Galois structures
on  a Galois  extension $L/K$  of order  $p^{2} q$,  where $p,  q$ are
distinct primes  and $p$  is odd, such  that $\Gamma=\Gal(L/K)$  has a
cyclic Sylow $p$-subgroup.  According to the discussion above (but see
Subsection~\ref{subsec:sylow}  for  a  technicality)  we  do  this  by
following  Byott's  approach,  that  is, by  determining  the  regular
subgroups  of $\Hol(G)$,  for each  group $G  = (G,  \cdot)$ of  order
$p^{2} q$, where $p, q$ are  distinct primes and $p$ is odd, such
  that  $G$  has  a  cyclic  Sylow $p$-subgroup.   This  is  in  turn
equivalent to  determining the right  skew braces $(G,  \cdot, \circ)$
such that $(G, \circ)\cong \Gamma$ (see~\cite[Problem 16]{problems}).

Our method relies on the use  of the alternate brace operation $\circ$
on $G$, mainly indirectly, that is through the use of the function
\begin{align*}
  \gamma :\ &G  \to \Aut(G)\\ &g \mapsto (x \mapsto  (x \circ g) \cdot
  g^{-1}),
\end{align*}
which is characterised by the functional equation
\begin{equation}
  \label{eq:GFE0}
  \gamma(g^{\gamma(h)} \cdot h) = \gamma(g) \gamma(h).
\end{equation}
(See  Theorem~\ref{thm:gamma-for-regular} and  the ensuing  discussion
for the  details.)  The functions  $\gamma$ satisfying~\eqref{eq:GFE0}
are  in  one-to-one  correspondence  with  the  regular  subgroups  of
$\Hol(G)$, and  occur naturally in  the theory of (skew)  braces: they
are called  $\mu$ in~\cite{Rump}, and  $\lambda$ in the  literature of
skew  braces~\cite{skew}.   They  have  been exploited,  albeit  in  a
somewhat different  context, in~\cite{affine, fgab, perfect,  p4}.  It
follows  that  to determine  the  number  $e'(\Gamma, G)$  defined  in
Theorem~\ref{th:byott96} we can count  the number of functions $\gamma
:\ G  \to \Aut(G)$  verifying \eqref{eq:GFE0} and  such that,  for the
operation $\circ$ defined on $G$ by
$$g\circ h=g^{\gamma(h)} h,$$ we have $(G,\circ)\cong\Gamma$.

In Section~\ref{sec:gamma} we develop methods to deal with these gamma
functions, that will  make enumerating them easier. As  a side effect,
our methods allow  us also to give alternative proofs  of some results
in the literature.  In Subsection~\ref{subs:nilpotent} we give a proof
of the  results of Byott~\cite[Theorem  1]{Byott13} about the  case of
finite   nilpotent   groups.    In  Subsection~\ref{subs:pq},   as   a
preliminary to our classification, we  give a compact treatment of the
case of groups of order $p q$, with $p, q$ distinct primes, dealt with
by   Byott   in~\cite{Byott04}  (see   also~\cite{AcriBonatto_pq}).    In
Subsection~\ref{subs:Kohl} we present a proof of a generalisation of a
result of Kohl~\cite{Ko13, Ko16}.

\subsection{Hopf-Galois structures of order $p^2q$}

Starting   with   Section~\ref{sec:the-groups},    we   restrict   our
consideration to the groups of order  $p^{2} q$, where $p$ and $q$ are
distinct primes; those  with cyclic Sylow $p$-subgroups  are listed in
Table~\ref{table:grp_aut}.

Applying  the results  of  Section~\ref{sec:gamma}, we  show that  for
$p>2$  a  Galois field  extension  $L/K$  with group  $\Gamma$  admits
Hopf-Galois structures  of type $G$ only  for those $G$ such  that $G$
and    $\Gamma$    have    isomorphic   Sylow    $p$-subgroups    (see
Theorem~\ref{teo:sylow} and Corollary~\ref{cor:t33}).

In Section~\ref{sec:proof} we  show that if $\Gamma$  has cyclic Sylow
$p$-subgroups and  $p > 2$,  each $G$ with cyclic  Sylow $p$-subgroups
defines  some  Hopf-Galois structure  on  $L/K$.   We then  explicitly
determine the number of structures for each type.

We accomplish this  by calculating, given a group $G  = (G, \cdot)$ of
order   $p^{2}  q$,   $p>2$,  the   following  equivalent   data  (see
Theorem~\ref{thm:gamma-for-regular}  and the  ensuing discussion)  for
each group $\Gamma$ of order $p^{2} q$ with cyclic Sylow $p$-subgroups
\begin{enumerate}
\item
  the total  number, and the  number and lengths of  conjugacy classes
  within $\Hol(G)$, of regular subgroups isomorphic to $\Gamma$;
\item
  the total number, and the  number of isomorphism classes, of (right)
  skew braces $(G, \cdot, \circ)$ such that $\Gamma \cong (G, \circ)$.
\end{enumerate}
\begin{remark}
  Here  and in  the following,  by  the \emph{(total)  number of  skew
    braces}  we  mean, given  a  group  $(G,  \cdot)$, the  number  of
  distinct operations  $\circ$ on  the set $G$  such that  $(G, \cdot,
  \circ)$ is a skew brace.
\end{remark}

The following theorems  summarise our results, where  the notation for
the groups is that given in Table~\ref{table:grp_aut}.
\begin{theo}
  \label{number}
  Let $L/K$ be a Galois field  extension of order $p^{2} q$, where $p$
  and $q$  are two  distinct primes with  $p > 2$,  and let  $\Gamma =
  \Gal(L/K)$.

  Let $G$ be a group of order $p^{2} q$.

  If the  Sylow $p$-subgroups  of $G$ and  $\Gamma$ are  not isomorphic,
  then there are no Hopf-Galois structures of type $G$ on $L/K$.

  If the  Sylow $p$-subgroups of $\Gamma$  and $G$ are cyclic,  then the
  numbers $e(\Gamma, G)$ of Hopf-Galois  structures of type $G$ on $L/K$
  are given in the following tables.
  \begin{enumerate}[(i)]
  \item For $q \nmid p-1$:
    \begin{center}
      \begin{tabular}{c | c c c}
        \diagbox[width=3.0em]{$\Gamma$}{$G$} & 1  & 2 & 3
        \\\hline
        1 & $p$ & $2 p (p-1)$ & $2 p (p-1)$
        \\
        2 & $p q$ & $2 p (p q - 2 q +
        1)$ & $2 p q (p - 1)$
        \\
        3 &  $p q$ & $2 p q (p - 1)$ & $2(p^{2}
        q - p q - q + 1)$ \\
      \end{tabular}
    \end{center}
    where the upper left sub-tables of sizes  $1\times 1$, $2\times 2$ and  $3\times 3$ give
    respectively the  cases $p \nmid q-1$,  $p \mmid q-1$ and  $p^2 \mid
    q-1$.
  \item For $q \mid p-1$:
    \begin{center}
      \begin{tabular}{c | c c}
        \diagbox[width=3.0em]{$\Gamma$}{$G$}
        & 1 & 4
        \\\hline
        1 & $p$ & $2p(q-1)$
        \\
        4 & $p^{2}$ & $2(p^{2}q-2p^2+1)$ \\
      \end{tabular}
    \end{center}
  \end{enumerate}

\end{theo}

\begin{theo}
  \label{number_sb}
  Let $G = (G, \cdot)$ be a group of order $p^{2} q$, where $p, q$ are
  distinct primes, with $p > 2$.

  If  $\Gamma$  is   a  group  of  order  $p^{2}  q$   and  the  Sylow
  $p$-subgroups  of  $G$ and  $\Gamma$  are  not isomorphic,  then  no
  regular subgroup of $\Hol(G)$ is isomorphic to $\Gamma$.
  
  If $G$ and  $\Gamma$ have both cyclic Sylow  $p$-subgroups, then the
  following tables give equivalently
  \begin{enumerate}
  \item
    the  number  $e'(\Gamma,  G)$  of  regular  subgroups  of  $\Hol(G)$
    isomorphic to $\Gamma$;
  \item
    the  number of  (right) skew  braces $(G,  \cdot, \circ)$  such that
    $\Gamma \cong (G, \circ)$.
  \end{enumerate}

  \begin{enumerate}[(i)]
  \item For $q \nmid p-1$:
    \begin{center}
      \begin{tabular}{c | c c c}
        \diagbox[width=3.0em]{$\Gamma$}{$G$}
        & 1  & 2 & 3
        \\\hline
        1 & $p$ & $2 p q$ & $2 q$
        \\
        2 & $p (p-1)$ & $2 p (p q - 2 q + 1)$ &
        $2 q (p - 1)$
        \\
        3 & $p^2  (p-1)$ & $2 p^2 q (p - 1)$ & $2(p^{2}
        q - p q - q + 1)$ \\
      \end{tabular}
    \end{center}
    where the upper left sub-tables of sizes  $1\times 1$, $2\times 2$ and  $3\times 3$ give
    respectively the  cases $p \nmid q-1$,  $p \mmid q-1$ and  $p^2 \mid
    q-1$.
  \item For $q \mid p-1$:
    \begin{center}
      \begin{tabular}{c | c c}
        \diagbox[width=3.0em]{$\Gamma$}{$G$} & 1 & 4 \\ \hline 1 & $p$ &
        $2 p^3$ \\ 4 & $q-1$ & $2 (p^{2}q-2p^2+1)$ \\
      \end{tabular}
    \end{center}
  \end{enumerate}
  
  For  each  group $\Gamma$  of  order  $p^{2}  q$ with  cyclic  Sylow
  $p$-subgroups, the following tables give equivalently
  \begin{enumerate}
  \item
    the  number  of  conjugacy  classes within  $\Hol(G)$  of  regular
    subgroups isomorphic to $\Gamma$;
  \item
    the  number of  isomorphism  classes of  skew  braces $(G,  \cdot,
    \circ)$ such that $\Gamma \cong (G, \circ)$.
  \end{enumerate} 
  \begin{enumerate}[(i)]
  \item For $q \nmid p-1$:
    \begin{center}
      \begin{tabular}{c | c c c}
        \diagbox[width=3.0em]{$\Gamma$}{$G$}
        & 1 & 2 & 3
        \\\hline
        1 & $2$ & $2p$ & $2$
        \\
        2 &  $p$ & $2p(p-1)$ & $2(p-1)$
        \\
        3 & $p$  & $2p(p-1)$ & $2p(p-1)$ \\
      \end{tabular}
    \end{center}
    where the upper left sub-tables of sizes $1\times 1$, $2\times 2$ and $3\times 3$ give
    respectively the cases $p \nmid q-1$,  $p \mmid q-1$ and $p^2 \mid
    q-1$.
  \item For $q \mid p-1$:
    \begin{center}
      \begin{tabular}{c|cc}
        \diagbox[width=3.0em]{$\Gamma$}{$G$}
        & 1 & 4
        \\\hline
        1 & $2$ & $4$
        \\
        4 & $1$ & $2(q-1)$ \\
      \end{tabular}
    \end{center}
  \end{enumerate}
\end{theo}
The   lengths  of   the   conjugacy  classes   are   spelled  out   in
Propositions~\ref{prop:G1},    \ref{prop:G3},    \ref{prop:G2},    and
\ref{prop:G4}.

From Theorem~\ref{number}  we obtain  the total number  of Hopf-Galois
structures.
\begin{corollary}
  Let $L/K$ be  a Galois extension of order $p^2q$,  where $p$ and $q$
  are two distinct primes with $p>2$, and let $\Gamma=\Gal(L/K)$.
  
  Assume that $\Gamma$ has cyclic Sylow $p$-subgroups.
  
  The total number of the Hopf-Galois structures on $L/K$ is given by
  the following table.
  \begin{center}
    \begin{tabular}{c|c|c}
      Conditions  &  \hphantom{   }$\Gamma$\hphantom{  }&  Hopf-Galois
      structures
      \\\hline
      $q\nmid p-1$,  $p \nmid  q-1$ &  1 &  $p$
      \\
      $q\nmid p-1$, $p \mmid q-1$  & 1 & $p(2p-1)$ \\
      $q\nmid p-1$,
      $p^2 \mid  q-1$& 1 &  $p(4p-3)$ \\ $q\mid  p-1$ & 1  & $p(2q-1)$
      \\\hline
      $q\nmid  p-1$,  $p  \mmid q-1$  &  2 &  $p(2pq-3q+2)$
      \\
      $q\nmid p-1$,  $p^2\mid q-1$&  2 &  $p(4pq-5q+2)$
      \\\hline
      $q\nmid  p-1$, $p^2\mid  q-1$&  3 &  $4p^2q-3pq-2q+2$
      \\\hline
      $q\mid p-1$ & 4 & $2p^2q-3p^2+2$ \\
    \end{tabular}
  \end{center}
\end{corollary}

\begin{corollary}
  \label{cor:ciclico}  A cyclic extension $L/K$ of degree $p^2q$, with
  $p>2$, admits exactly $p$ 
  cyclic Hopf-Galois structures.

  Moreover, the cyclic structures  are the only Hopf-Galois structures
  on $L/K$ if and only if there is only one isomorphism type of groups
  of order $p^2q$  with cyclic Sylow $p$-subgroup, namely  if and only
  if $p \nmid q-1$ and $q \nmid p-1$.
\end{corollary}

Note that the first statement of the above corollary follows also from
\cite{Byott13}.

Finally,  in  Section~\ref{sec:split}  we apply  the  method  proposed
in~\cite{CRV16} to count the cyclic Hopf-Galois structures on a Galois
extension of order $p^{2} q$. In fact,  in this case one can show that
all cyclic structures are induced, and the induced structures are easy
to compute, so that we recover our results in Theorem~\ref{number} for
the group  $G$ of type  1; note that  the only induced  structures for
extensions of order $p^{2} q$ are the abelian ones.

The case of groups of order $p^{2} q$  with $p < q - 1$ has been dealt
with in~\cite{Dietzel}  for braces. Braces have been introduced by
Rump in~\cite{braces}, and  correspond to
the case when the group $(G, \cdot)$ is abelian.

In a separate  paper, we intend to deal with  the regular subgroups of
holomorphs of groups of order  $p^{2} q$ with elementary abelian Sylow
$p$-subgroups, where $p,  q$ are distinct primes, and of  order $4 q$,
where $q$ is an odd prime.
\begin{remark}
  \label{rem:two-wont-do}
  The reason why we are not treating here  the case $p = 2$ is that in
  the  holomorph of  a group  $G$  of order  $4 q$  with cyclic  Sylow
  $2$-subgroups there  are regular  subgroups with  elementary abelian
  Sylow  $2$-subgroups.  This  boils  down  to the  fact  that in  the
  normaliser of  the cyclic subgroup $\Span{(1234)}$  in the symmetric
  group on four  letters there is a  regular subgroup $\Span{(13)(24),
    (12)(34)}$,     which      is     elementary      abelian.      In
  Subsection~\ref{subsec:sylow} we show that this does not happen when
  $p$ is odd.
\end{remark}

\subsection{Another classification}

After this manuscript was  submitted,  Acri and Bonatto posted on
arXiv a manuscript~\cite{AcriBonatto_p2q}, in which they employ the methods
of~\cite{skew}  to count the isomorphism  classes of
skew braces of size $p^{2} q$ for odd primes $p$ and $q$. In
Subsection~\ref{subs:AB} below, we compare their results and methods,
and more generally the methods employed in the theory of skew braces,
with ours.

\section*{Acknowledgements}
The  authors are  most grateful  to  Nigel Byott  for suggesting  this
problem.

The authors are grateful to the referee for the useful comments.

The system for computational discrete algebra \textsl{GAP}~\cite{GAP4}
has been invaluable in gaining insight into this subject.

\section{The gamma function}
\label{sec:gamma}

\subsection{Describing the regular subgroups of the holomorph}

There are several equivalent ways to describe the regular subgroups of
the  holomorph  of  a  group   $(G,  \cdot)$.   As  explained  in  the
Introduction, we will appeal to  the alternate group operation $\circ$
which occurs in  the definition of a skew  brace~\cite{skew}, but then
mainly  use  the  equivalent   method  of  gamma  functions,  employed
in~\cite{affine, fgab, perfect, p4}.

The \emph{abstract holomorph} of a group $G$ is the natural semidirect
product $\Aut(G) G$. In this paper we will use a certain concrete realisation
of it, namely the \emph{permutational holomorph}, as defined in the
following.

Given a group $G$, denote by $S(G)$ the group of all permutations on
the underlying set $G$. The \emph{right regular representation} of $G$ is the
homomorphism
\begin{align*}
  \rho :\ &G \to S(G)\\
  &g \mapsto (x \mapsto x g).
\end{align*}
Similarly, the \emph{left regular representation} of $G$ is the
\emph{anti}homomorphism
\begin{align*}
  \lambda :\ &G \to S(G)\\
  &g \mapsto (x \mapsto g x).
\end{align*}
Write $\inv : g \mapsto g^{-1}$ for the inversion map on $G$. Clearly
$\inv \in S(G)$.

The following facts are  well known, see for instance~\cite[Lemma
  3.8]{PrSch}.
\begin{proposition}\ppar
  \label{prop:right-and-left}
  \begin{enumerate}
  \item\label{lemma:AutG-is-stabilizer} The stabiliser of $1$ in  
    the normaliser $N_{S(G)}(\rho(G))$ of the image $\rho(G)$ of the
    right regular representation is $\Aut(G)$. 
  \item\label{item:normaliser}  We have 
    \begin{equation*}
      N_{S(G)}(\rho(G))  =  \Aut(G)   \rho(G)  =  \Aut(G)  \lambda(G)  =
      N_{S(G)}(\lambda(G)),
    \end{equation*}
    and this group is  isomorphic to the abstract holomorph  $\Aut(G) G$
    of $G$.
  \item\label{item:inversion}  $\inv$
    centralises $\Aut(G)$,  and
    conjugates $\rho(G)$  to $\lambda(G) \le \Hol(G)$,
    that is
    \begin{equation*}
      \rho(G)^{\inv} = \lambda(G). 
    \end{equation*}
    Thus $\inv$ normalises $N_{S(G)}(\rho(G))$. 
  \end{enumerate}
\end{proposition}

In the following we will refer to $N_{S(G)}(\rho(G))$ as the
\emph{(permutational) holomorph} of $G$, and denote it by $\Hol(G)$.

Let  $G$   be  a  finite  group,   and  let $N  \le  \Hol(G)$ be  a  regular
subgroup. Since $N$ is regular, the map $N \to G$ sending $n \in N$ to
$1^{n}$ is  a bijection.  Thus for each  $g \in G$  there is  a unique
element $\nu(g) \in N$, such that $1^{\nu(g)}  = g$, that is, $\nu : G
\to  N$ is  the  inverse of  $n \mapsto  1^{n}$.  Now  $1^{\nu(g)
  \rho(g)^{-1}} =  1$, so  that $\nu(g)  \rho(g)^{-1} \in  \Aut(G)$ by
Proposition~\ref{prop:right-and-left}\eqref{lemma:AutG-is-stabilizer}.
Therefore for $g \in G$ we can write $\nu(g)$ uniquely in the form
\begin{equation}\label{eq:unique-form}
  \nu(g) = \gamma(g) \rho(g),
\end{equation}
for a suitable map $\gamma : G \to \Aut(G)$. We have
\begin{equation}\label{eq:nu}
  \nu(g) \nu(h)
  =
  \gamma(g) \rho(g) \gamma(h) \rho(h)
  =
  \gamma(g) \gamma(h) \rho(g^{\gamma(h)} h).
\end{equation}
Since $N$ is a subgroup of $S(G)$, $\gamma(g) \gamma(h) \in \Aut(G)$, and the
expression~\eqref{eq:unique-form} is unique, we have
\begin{equation*}
  \gamma(g) \gamma(h) \rho(g^{\gamma(h)} h) = \gamma( g^{\gamma(h)}
  h)\rho(g^{\gamma(h)} h),
\end{equation*}
from which we obtain
\begin{equation}\label{eq:gamma}
  \gamma( g^{\gamma(h)} h )
  =
  \gamma(g) \gamma(h).
\end{equation}
We obtain
\begin{theorem}
  \label{thm:gamma-for-regular}
  Let $(G, \cdot)$ be a finite group. The following data are equivalent.
  \begin{enumerate}
  \item\label{item:regular}
    A regular subgroup $N \le \Hol(G)$.
  \item\label{item:gamma}
    A map $\gamma : G \to \Aut(G)$ such that
    \begin{equation}\label{eq:gamma-for-circ}
      \gamma( g^{\gamma(h)} h )
      =
      \gamma(g) \gamma(h).
    \end{equation}
  \item\label{item:braces} 
    A group operation $\circ$ on $G$ such that
    for $g, h, k \in G$ 
    \begin{equation}
      \label{eq:brace-axiom}
      (g h) \circ k = (g \circ k) k^{-1} (h \circ k),
    \end{equation}
    that is, such that $(G, \cdot, \circ)$ is a (right) skew brace.
  \end{enumerate}
  The data of~\eqref{item:regular}-\eqref{item:braces} are related as follows. 
  \begin{enumerate}[(i)]\label{item:circ}
  \item
    $
    g \circ h = g^{\gamma(h)} h
    $
    for $g, h \in G$. 
  \item 
    Each element of $N$ can be written uniquely in the form $\nu(h) =
    \gamma(h) \rho(h)$, for some $h \in G$.
  \item 
    For $g, h \in G$ one has
    $
    g^{\nu(h)} = g \circ h.
    $
  \item\label{item:gamma-is-hom}
    The map
    \begin{align*}
      \gamma:  (G, \circ)  \to  \Aut(G)
    \end{align*}
    is a morphism.    
  \item\label{item:nu-is-iso}
    The map 
    \begin{align*}
      \setlength\arraycolsep{1.5pt}
      \begin{matrix}
        \nu: & (G, \circ) & \to & N\\
        & h &\mapsto &\gamma(h) \rho(h)
      \end{matrix}
    \end{align*}
    is an isomorphism.    
  \end{enumerate}
\end{theorem}
This is basically~\cite[Theorem 1.2]{p4}. The equivalence of~\eqref{item:braces}
to  the  other  items of  Theorem~\ref{thm:gamma-for-regular}  follows
from~\cite[Theorem 4.2]{skew};  for the convenience of  the reader, we
provide  the table below detailing the  relations between  the properties  of
$\gamma$ and the brace axioms. In this table, $(G, \cdot)$ is
a group, $\circ$  is an operation on  $G$, and for each $g  \in G$ we
define a  function $\gamma(g) :  G \to G$ by  $x \mapsto (x  \circ g)
\cdot g^{-1}$, so that $x \circ g = x^{\gamma(g)} \cdot g$.

\begin{center}
  \begin{tabular}{p{5.75cm}|p{5.75cm}}
    \multicolumn{1}{c|}{Property of $\circ$}
    &
    \multicolumn{1}{c}{Property of $\gamma$}
    \tabularnewline\hline\hline
    \raggedright
    The brace axiom~\eqref{eq:brace-axiom} of
    Theorem~\ref{thm:gamma-for-regular} holds
    &
    \raggedright
    $\gamma(g)$ is an endomorphism of $G$, for each $g \in G$
    \tabularnewline\hline
    $\circ$ is associative
    &
    \raggedright
    $\gamma$ satisfies~\eqref{eq:gamma-for-circ}
    of Theorem~\ref{thm:gamma-for-regular}
    \tabularnewline\hline
    $\circ$ admits inverses
    &
    \raggedright
    $\gamma(g)$ is bijective, for each $g \in G$
  \end{tabular}
\end{center}
In the  table, the properties  on the  first line are  equivalent. The
properties on  the second  line are  equivalent, under  the assumption
that $\gamma(g)$  is an endomorphism of  $G$, for each $g  \in G$.  On
the third line, the property on  the right implies the property on the
left,  while  to  prove  the left-to-right  implication  one  need  to
assume~\eqref{eq:gamma-for-circ}                                    of
Theorem~\ref{thm:gamma-for-regular}. The fact that $(G, \circ)$ has an
identity follows from the properties in the first line.

\begin{remark}
  In the rest of the paper, every time we discuss a regular subgroup
  as in~\eqref{item:regular} of Theorem~\ref{thm:gamma-for-regular},
  or a $\gamma$ as in~\eqref{item:gamma}, we will employ the rest of
  the notation of Theorem~\ref{thm:gamma-for-regular} without further mention.
\end{remark}

\begin{definition}
  \label{def:GF}
  Let $G$ be a group, $A \le G$, and $\gamma: A \to \Aut(G)$ a
  function.

  $\gamma$ is said to satisfy the \emph{gamma functional
    equation} (or \emph{GFE} for short) if
  \begin{equation*}
    \gamma( g^{\gamma(h)} h )
    =
    \gamma(g) \gamma(h),
    \qquad
    \text{for all $g, h \in A$.}
  \end{equation*}
  
  $\gamma$ is said to be a \emph{relative gamma function} (or
  \emph{RGF} for short) on $A$ if it satisfies the gamma functional
  equation, and $A$ is $\gamma(A)$-invariant.

  If $A = G$, a relative gamma function is simply called a
  \emph{gamma function} (or \emph{GF} for short) on $G$.
\end{definition}
A RGF $\gamma : A \to \Aut(G)$ defines a group operation
$g \circ h =  g^{\gamma(h)} h$ on $A$.

\subsection{A comparison to other results and methods}
\label{subs:AB}

As mentioned in the Introduction, Acri and Bonatto have posted on
arXiv a manuscript~\cite{AcriBonatto_p2q} whose results overlap with
ours. They enumerate the isomorphism  classes of
skew braces of size $p^{2} q$, for odd primes $p$ and $q$, including
thus the case when the Sylow $p$-subgroups are elementary abelian,
which we do not cover in this paper. Equivalently, they enumerate the
conjugacy classes of regular subgroups in the holomorph of groups $G$ of
order $p^{2} q$, for odd primes $p$ and $q$. We provide in addition
the sizes of the conjugacy classes, and include the case $q = 2$. The
common results 
of~\cite{AcriBonatto_p2q} agree with ours.

They organise their classification, according to the algorithm
of~\cite{skew}, which they already employed in~\cite{AcriBonatto_pq},
on the basis 
of  the size  of  what in  our  context is  the  image $\gamma(G)  \le
\Aut(G)$ of the gamma function.

We have already noted that in the literature one usually deals with
\emph{left} skew braces, whereas we use the \emph{right} ones; as
mentioned in the Introduction, we translate results and methods in the
literature into our notation.

In the theory of skew braces $(G, \cdot, \circ)$, a morphism $\lambda :
(G, \circ) \to \Aut(G, \cdot)$ is defined, which is exactly our
$\gamma$. (We reserve $\lambda$ for the left regular representation.)

Several   intermediate  results   of~\cite{AcriBonatto_p2q}  naturally
overlap with  ours, for instance  the fact that $\ker(\gamma)  \le (G,
\circ)$. In our context, an ideal of a skew brace $(G, \cdot, \circ)$
is a subgroup $H \le (G, \cdot)$ which is $\gamma(G)$-invariant; some
results of~\cite{AcriBonatto_p2q} on $\gamma(G)$-invariant Sylow
subgroups are also similar to ours.

\begin{remark}
  Our $\gamma$ are related to the bijective $1$-cocycles
  of~\cite{skew}. Recall that if the group $(H, \bullet)$ acts (on the
  right) on the group $(G,
  \cdot)$ via automorphisms, a map $\pi : H \to G$ is said to be a
  $1$-cocycle if
  \begin{equation*}
    \pi(a \bullet b)
    =
    \pi(a)^{b} \cdot \pi(b),
    \quad
    \text{for $a, b \in H$,}
  \end{equation*}
  where $g^{h}$ denotes the action of $h \in H$ on $g \in G$.
  It is proved in~\cite[Proposition 1.11]{skew} that if $(G, \cdot,
  \circ)$ is a skew brace, then the 
  identity map $\pi : (G, \circ) \to (G, \cdot)$ is a $1$-cocycle, where $G$
  acts on itself, in our notation, by $g^{h} = g^{\gamma(h)}$. In fact
  \begin{equation*}
    \pi(a \circ b)
    =
    a \circ b
    =
    a^{\gamma(b)} \cdot b
    =
    \pi(a)^{b} \cdot \pi(b).
  \end{equation*}
  Conversely,
  a  bijective
  $1$-cocycle induces a skew brace structure $(G, \cdot, \circ)$ via
  \begin{equation*}
    a \circ b = \pi ( \pi^{-1}(a) \bullet \pi^{-1}(b) ),
    \quad
    \text{for $a, b \in G$.}
  \end{equation*}
  In our notation, this means that the function
  \begin{align*}
    \gamma :\ &G \to \Aut(G)\\
    &b \mapsto (a \mapsto \pi ( \pi^{-1}(a) \bullet
    \pi^{-1}(b) ) \cdot b^{-1})
  \end{align*}
  is a GF.
\end{remark}

\subsection{Invariant subgroups}

The  following  proposition is  a  slightly  more general  version  of
\cite[Proposition   3.3]{Tsang2019}

\begin{prop}
  \label{prop:2su3}
  Let $G$ be a finite group, let $H \subseteq G$ and let $\gamma$ be a
  GF on $G$.
  
  Any  two of  the
  following conditions imply the third one:
  \begin{enumerate}
  \item\label{i1} $H \le G$;
  \item\label{i2} $(H,\circ) \le (G, \circ)$;
  \item\label{i3} $H$ is $\gamma(H)$-invariant.
  \end{enumerate}
  If these conditions hold, then $(H,\circ)$ is isomorphic to a regular
  subgroup of $\Hol(H).$ 
\end{prop}
\begin{proof}
  The equivalence follows from the fact that for $h_1,h_2\in H$ we have
  \begin{equation*}
    h_{1} \circ h_{2} = h_{1}^{\gamma(h_{2})} h_{2}
    \qquad\text{and}\qquad
    h_{1}^{\gamma(h_{2})^{-1}} \circ h_{2} = h_{1} h_{2}.
  \end{equation*}
  If the conditions hold, the fact that  $H$ is $\gamma(H)$-invariant
  implies that we have a map
  \begin{align*}
    \gamma' :\ &H \to \Aut(H)\\
    &h \mapsto \gamma(h)_{\restriction H}
  \end{align*}
  which satisfies the GFE because  $\gamma$ does, so that $\gamma'$ is
  a GF on $H$.  
  If $\circ'$ is
  the group operation on $H$ associated to $\gamma'$, the identity map
  \begin{equation*}
    (H, \circ') \to (G, \circ)
  \end{equation*}
  is a group morphism, as for $h_1, h_2 \in H$ one has
  \begin{equation*}
    h_1 \circ' h_2 = h_1^{\gamma'(h_2)} h_2 = h_1^{\gamma(h_2)} h_2 = h_1 \circ h_2.
  \end{equation*}
  It follows that $(H, \circ')=(H,\circ)$ is isomorphic to a regular subgroup  of $\Hol(H).$
\end{proof}

The subgroups $H$ verifying the conditions of Proposition~\ref{prop:2su3} in the language of braces are called {\it sub-skew braces}.
\subsection{An application: nilpotent groups}
\label{subs:nilpotent}

The following result is due to N.~P.~Byott.
\begin{theorem}[\protect{\cite[Theorem 1]{Byott13}}]
  Let $\Gamma$ be a finite nilpotent group of order $n$.

  Then for each nilpotent group $G$ of order $n$ we have
  \begin{equation*}
    e(\Gamma, G)
    =
    \prod_{p} e(\Gamma_{p}, G_{p}),
  \end{equation*}
  where  $p$ ranges  over the  primes dividing  $n$, and  $\Gamma_{p}$
  resp.~$G_{p}$ denote the Sylow $p$-subgroups of $\Gamma$ resp.~$G$.
\end{theorem}

We now give an alternate proof of Byott's result, using gamma
functions.

\begin{proof}
  Let $p_{1},  \dots, p_{k}$ be the distinct primes dividing $n$.
  The  finite nilpotent  group $G$  is the  direct product  of its
  distinct Sylow subgroups, each of which is characteristic in $G$, so we
  have
  \begin{equation}
    \label{eq:aut_prod}
    \Aut(G) = \prod_{i=1}^{k} \Aut(G_{p_{i}}).
  \end{equation}
  and the same holds for $\Gamma$.
  Therefore, in view of \eqref{eq:e_vs_e-prime}, we can rephrase
  Byott's result as 
  \begin{equation}
    \label{eq:Byott_revisited}
    e'(\Gamma, G)
    =
    \prod_{i=1}^{k} e'((\Gamma_{p_{i}}, G_{p_{i}}),
  \end{equation}
  As we already noted in the Introduction
  \begin{equation*}
    e'(\Gamma, G)
    =
    \Size{ \Set{ \text{$\gamma$ GF on $G$} : (G,\circ)\cong\Gamma }}.
  \end{equation*}

  Let now $\gamma$ be a gamma function on $G$ such that
  $(G,\circ)\cong\Gamma$. Let 
  $p, q$ be distinct primes dividing $n$, and let $a$ be a
  $p$-element, and $b$ a $q$-element of $G$ and $(G, \circ)$. Since
  $G$ and $(G, \circ)$ are nilpotent, $a$ and $b$ commute in both
  groups. We thus have
  \begin{equation*}
    a^{\gamma(b)} b = a \circ b = b \circ a = b^{\gamma(a)} a = a b^{\gamma(a)},
  \end{equation*}
  which implies $b^{\gamma(a)} =
  b$ (and $a^{\gamma(b)} = a$).

  This shows that  the  image  $\gamma(G_{p})$   of  the  Sylow
  $p$-subgroup of $G$  under  $\gamma$ acts  trivially  on the  Sylow
  $q$-subgroups  of  $G$,  for  $q  \ne p$.  The  composition  of  the
  restriction of  $\gamma$ to $G_{p}$,  followed by the  projection of
  $\Aut(G)$  onto  $\Aut(G_{p})$,  is  clearly  a  gamma  function  on
  $G_{p}$.

  It follows that every gamma function $\gamma$ on $G$ is obtained as
  \begin{equation*}
    \gamma(x)
    =
    \gamma_{1}(x_{1}) \dots \gamma_{k}(x_{k}),
  \end{equation*}
  where $x$  is uniquely written  as $x_{1} \cdots x_{k}$,  with $x_{i}
  \in G_{p_{i}}$, and $\gamma_{i}  : G_{p_{i}} \to \Aut(G_{p_{i}}) \le
  \Aut(G)$    is   a    gamma   function    on   $G_{p_{i}}$.    
  Proposition~\ref{prop:2su3} now implies  that  the Sylow
  $p$-subgroup 
  $(G, \circ)_{p}$ of $(G, \circ)$ is $(G_{p}, \circ)$, which is then
  isomorphic to $\Gamma_p$.  
  This proves equality \eqref{eq:Byott_revisited}.
\end{proof}

\subsection{Isomorphism of Sylow $p$-subgroups}

From Proposition~\ref{prop:2su3} we have that if one of the  Sylow
$p$-subgroups $H$ 
of $G$  is invariant  under $\gamma(H)$, then  the isomorphism
type of the Sylow  $p$-subgroups of $(G,\circ)$ is to be  found among
the  isomorphism types of regular  subgroups of 
$\Hol(H)$. If there is no such invariant Sylow $p$-subgroup, then
it is conceivable that the  Sylow $p$-subgroups of regular subgroups
of    $\Hol(G)$    could    take    further    isomorphism    types.

For some classes of $p$-subgroups the criterion given in
Proposition~\ref{prop:2su3} can be made more explicit. 

\begin{corollary}
  \label{cor:sylow-gen}
   \label{cor:e=0}
   Let  $G$ and $\Gamma$ be finite  groups. Suppose $e'(\Gamma,G)\ne0$. Let $N\le\Hol(G)$ be a regular subgroup isomorphic to $\Gamma$, and $\gamma$ the GF associated to $N$. 
   
   If $p$ is  an odd prime and $H$ is a $\gamma(H)$-invariant $p$-subgroup of
  $G$, then 
  \begin{enumerate}
  \item if $H$ is cyclic then $H \cong
    (H, \circ)$;
  \item  if $H$ is abelian and of rank  $m$, with $m < p - 1$, or
    $m = 2$ and  $p = 3$, and $(H,\circ)$ is abelian  too, then $H \cong
    (H, \circ)$.
  \end{enumerate}  
\end{corollary}
\begin{proof}
  If $H$ is a cyclic $p$-group with $p > 2$, by~\cite[Corollary,
  ~680]{Rump} or by~\cite[Proposition 4]{CreSal2019},  every regular
  subgroup of $\Hol(H)$ is cyclic (see also~\cite[(8.6)
  Proposition]{ChildsBook}).  

  In the case when $H$ is abelian of rank $m$, with $m + 1 < p$, or $m
  = 2$
  and $p = 3$, by~\cite[Theorem  1 and Proposition  4]{FCC12}, all the
  abelian subgroups of  $\Hol(H)$ are isomorphic to $H$.  
  
  Both statements now follow  from Proposition~\ref{prop:2su3}.
\end{proof}

In Theorem~\ref{teo:sylow} we will show  that in the case of groups of
order  $p^{2}  q$   with  $p  >  2$  the   conditions  of  Proposition
\ref{prop:2su3} are fulfilled for $H$  a Sylow $p$-subgroup of $G$ and
for each GF.   This will simplify our classification  since it implies
that $e(\Gamma,G)=0$ whenever the  Sylow $p$-subgroups of $\Gamma$ and
$G$    are    not   isomorphic    (we    have    already   noted    in
Remark~\ref{rem:two-wont-do} that this does not  hold for $p = 2$). We
will  also prove  (see Theorem~\ref{number})  that in  that case  the
condition of having isomorphic  Sylow $p$-subgroups is also sufficient
for $e(\Gamma,G)\ne 0$.

\subsection{Tools of the trade}

We now collect several facts related to the gamma functions, which we
are going to exploit for our classification.

We begin by rephrasing conjugacy of regular subgroups within $\Hol(G)$
in terms of gamma functions, and recording a simple property that will
be useful in Section~\ref{sec:proof}. Conjugacy of regular subgroups of the
holomorph of the group $(G, \cdot)$ is equivalent to isomorphism of skew
braces $(G, \cdot, \circ)$, and to a certain equivalence of Hopf-Galois structures.

Note that since $N$ is regular, we have $\Hol(G) = \Aut(G) N$, so that
the conjugates of $N$ under 
$\Hol(G)$ coincide with the conjugates under $\Aut(G)$.
Now
\cite[Proposition A.3]{SV2018} (see also the last statement of~\cite[Theorem
  1]{affine}) states that, given a group $G = (G, \cdot)$, there is a bijection
between isomorphism classes of skew braces $(G, \cdot, \circ)$,
and classes of regular subgroups of $\Hol(G)$ under conjugation by
elements of $\Aut(G)$.  We give the simple translation of the
latter in terms of gamma functions.
\begin{lemma}
  \label{lemma:conjugacy}
  Let $G$ be a group, $N$ a regular subgroup of $\Hol(G)$, and
  $\gamma$ the associated gamma function.

  Let $\beta \in \Aut(G)$.
  \begin{enumerate}
  \item 
    The gamma function $\gamma^{\beta}$ associated to the regular subgroup
    $N^{\beta}$  is given by
    \begin{equation}
      \label{eq:conjugacy}
      \gamma^{\beta}(g)
      =
      \gamma(g^{\beta^{-1}})^{\beta}
      =
      \beta^{-1} \gamma(g^{\beta^{-1}}) \beta,
    \end{equation}
    for $g \in G$.
  \item
    \label{item:invariance-under-conjugacy}
    If $H \le G$ is invariant under $\gamma(H)$, then $H^{\beta}$ is
    invariant under $\gamma(H^{\beta})$.
  \end{enumerate}  
\end{lemma}

We will refer to the action of $\Aut(G)$ on $\gamma$ of the
Lemma as \emph{conjugation}.

\begin{proof}
  For  $x   \in  G$   we  have  $\nu(x)^{\beta}   =  \gamma(x)^{\beta}
  \rho(x^{\beta})$.  Since $\nu(x)^{\beta}$  takes $1$ to $x^{\beta}$,
  we  have  $\gamma^{\beta}(x^{\beta})   =  \gamma(x)^{\beta}$,  which
  yields~\eqref{eq:conjugacy} substituting $g = x^{\beta}$.

  If $H \le G$ is invariant under $\gamma(H)$, and $h_{1}, h_{2} \in
  H$, then
  \begin{equation*}
    (h_{1}^{\beta})^{\gamma^{\beta}(h_{2}^{\beta})}
    =
    h_{1}^{\beta \beta^{-1} \gamma(h_{2}) \beta}
    =
    (h_{1}^{\gamma(h_{2})})^{\beta}
    \in
    H^{\beta}.
  \end{equation*}
\end{proof}

We now record two simple facts, which we will be using repeatedly, 
concerning inverses and conjugacy in the group $(G, \circ)$ of 
Theorem~\ref{thm:gamma-for-regular}. We write
$a^{\ominus 1}$ for the inverse of $a \in G$ in $(G, \circ)$.

\begin{lemma}
  \label{lemma:inverse_and_conjugacy}
  In the notation of Theorem~\ref{thm:gamma-for-regular}, we have, for
  $a, b \in G$,
  \begin{equation*}
    a^{\ominus 1}
    =
    a^{- \gamma(a)^{-1}},
  \end{equation*}
  and
  \begin{equation*}
    a^{\ominus 1} \circ b \circ a
    =
    a^{-\gamma(a)^{-1} \gamma(b) \gamma(a)}  b^{\gamma(a)} a.
  \end{equation*}
\end{lemma}

\begin{proof}
  If $z$ is the inverse of $a$ in $(G, \circ)$, we have $1 = z \circ a
  = z^{\gamma(a)} a$, whence $z = a^{- \gamma(a)^{-1}}$.

  \begin{align*}
    a^{\ominus 1} \circ b \circ a
    &=
    a^{- \gamma(a)^{-1}} \circ b \circ a
    \\&=
    (a^{- \gamma(a)^{-1} \gamma(b)}  b) \circ a
    \\&=
    a^{-\gamma(a)^{-1} \gamma(b) \gamma(a)} b^{\gamma(a)} a.
  \end{align*}
\end{proof}

\begin{remark}
  Note, for later usage, that~\eqref{eq:gamma-for-circ} can be
  rephrased, setting $k = g^{\gamma(h)}$, as
  \begin{equation}\label{eq:gamma-for-dot}
    \gamma(k h) = \gamma(k^{\gamma(h)^{-1}}) \gamma(h).
  \end{equation}
\end{remark}

\begin{lemma}
  \label{lemma:kergamma}
  Let $G$ be a finite group, and $\gamma$ a GF on $G$.
  We have 
  \begin{enumerate}
  \item 
    $\ker(\gamma) \norm (G, \circ)$, and 
  \item $\ker(\gamma) \le G$.
  \end{enumerate}
\end{lemma}

\begin{proof}
  The first claim is  clear as $\gamma : (G, \circ)  \to \Aut(G)$ is a
  morphism (Theorem~\ref{thm:gamma-for-regular}\eqref{item:gamma-is-hom}).

  Since $\ker(\gamma)$ is invariant under
  $\gamma(\ker(\gamma)) = \Set{1}$, Proposition~\ref{prop:2su3} implies
  $\ker(\gamma) \le G$.
\end{proof}

In the statement of the next result we write $[g, \alpha] =
g^{-1} g^{\alpha}$, for $g \in G$, $\alpha \in \Aut(G)$. This is indeed
an ordinary commutator in the abstract holomorph of $G$.
We write 
\begin{equation*}
[A, \gamma(A)]
=
\Set{[x, \gamma(y)] : x, y \in A}.
\end{equation*}
\begin{lemma}
  \label{Lemma:gamma_morfismi}
  Let $G$ be a finite group, $A \le G$, and $\gamma : A \to \Aut(G)$ a
  function such that $A$ is invariant under $\gamma(A)$.

  Then any two of the following conditions imply the third one.
  \begin{enumerate}
  \item $\gamma([A, \gamma(A)]) = \Set{1}$.
  \item $\gamma : A \to \Aut(G)$ is a morphism of groups.
  \item $\gamma$ satisfies the GFE.
  \end{enumerate}
\end{lemma}

\begin{proof}
  Suppose $\gamma([A, \gamma(A)]) = \Set{1}$. If $\gamma$ is a
  morphism,  then for $x, y \in A$ we have
  \begin{equation*}
    \gamma(x^{\gamma(y)} y)
    =
    \gamma(x [x, \gamma(y)] y)
    =
    \gamma(x) \gamma([x, \gamma(y)]) \gamma(y)
    =
    \gamma(x) \gamma(y),
  \end{equation*}
  that is, $\gamma$ satisfies the GFE.
  Conversely, suppose $\gamma$ satisfies the GFE, so that
  $\gamma(A)$ is a subgroup of $\Aut(G)$. Then for $x, y \in A$ we have
  \begin{align*}
    \gamma(x y)
    &=
    \gamma(x^{\gamma(y)^{-1}}) \gamma(y)
    \\&=
    \gamma(x [x, \gamma(y)^{-1}]) \gamma(y)
    \\&=
    \gamma(x^{\gamma([x, \gamma(y)^{-1}])^{-1}}) \gamma([x,
      \gamma(y)^{-1}]) \gamma(y)
    \\&=
    \gamma(x) \gamma(y).
  \end{align*}
  Suppose now $\gamma$ is a morphism and satisfies the GFE. Then for
  $x, y \in A$ we have 
  \begin{equation*}
    \gamma(x) \gamma(y)
    =
    \gamma(x^{\gamma(y)} y)
    =
    \gamma(x [x, \gamma(y)] y)
    =
    \gamma(x) \gamma([x, \gamma(y)]) \gamma(y),
  \end{equation*}
  so that $\gamma([x, \gamma(y)]) = 1$.
\end{proof}

\subsection{Lifting and restriction}
\label{subsec:sylow}

The next result can be considered as a vestigial form of the First
Isomorphism Theorem for gamma functions.

We write 
\begin{align*}
\iota \colon &G\to\Aut(G)\\
&g \mapsto (x\mapsto g^{-1} x g).
\end{align*}
\begin{prop}
  \label{prop:1.4}
  Let $G$ be a finite group and  let $A$, $B$ be subgroups of $G$ such
  that    $G = A B$.    

  If $\gamma$ is a GF on $G$, and $B \le \ker(\gamma)$, then
  \begin{equation}    
    \label{eq:a-to-b}
    \gamma(a b) = \gamma(a), \text{ for } a \in A, b \in B,
  \end{equation}
  so that $\gamma(G) = \gamma(A)$.
  
  Moreover, if $A$  is  $\gamma(A)$-invariant, then 
  \begin{equation}
    \label{eq:gamma'}
    \gamma'=\gamma_{\restriction A}\colon A\to\Aut(G)
  \end{equation}  
  is a RGF on $A$ and  $\ker(\gamma)$  is  invariant  under  the
  subgroup   
  \begin{equation*}
    \Set{ \gamma'(a) \iota(a) : a \in A }
  \end{equation*}
  of   $\Aut(G)$.

  Conversely, let $\gamma':A\to \Aut(G)$ be a RGF such that
  \begin{enumerate}
  \item
    \label{item:cap}
    $\gamma'(A \cap B) \equiv 1$,
  \item
    \label{item:g-i}
    $B$ is invariant under $\{ \gamma'(a) \iota(a) : a \in A \}$.
  \end{enumerate}
  Then the map 
  \begin{equation*}
    \gamma(ab) = \gamma'(a), \text{ for } a \in A, b \in B,
  \end{equation*}
  is a well defined GF on $G$, and $\ker(\gamma) = \ker(\gamma') B$.
\end{prop}
In this  situation we will  say that  $\gamma$ is a  \emph{lifting} of
$\gamma'$.

\begin{proof}
  Clearly $\gamma$ is constant on  the cosets of $\ker(\gamma)$, and thus also
  on  the  cosets  of  $B$, so that~\eqref{eq:a-to-b}  holds,  and thus
  $\gamma(G) = \gamma(A)$. Assume  now that $A$  is $\gamma(A)$-invariant;
  by Proposition~\ref{prop:2su3}, $A$ is a subgroup of $(G,\circ)$ and 
  $\gamma'$ as in (\ref{eq:gamma'}) satisfies the GFE, so that
  $\gamma'$ is a RGF. 

  For $a \in A$ and $k \in \ker(\gamma)$ we have
  \begin{equation*}
    a^{\ominus 1} \circ k \circ a 
    = a^{-\gamma(a)^{-1} \gamma(k) \gamma(a)} k^{\gamma(a)} a
    = a^{-1} k^{\gamma(a)} a
    = k^{\gamma(a) \iota(a)},
  \end{equation*}
  and  since  $a^{\ominus  1}\circ k\circ  a\in\ker(\gamma)$,  we  get
  $k^{\gamma(a) \iota(a)}\in \ker(\gamma)$,  namely, $\ker(\gamma)$ is
  invariant under the action of $\Set{\gamma'(a) \iota(a) : a\in A}$.

  Note that the latter is a subgroup of $\Aut(G)$, as for $a_{1},
  a_{2} \in A$ we have
  \begin{align*}
    \gamma(a_{1}) \iota(a_{1}) \gamma(a_{2}) \iota(a_{2})
    &=
    \gamma(a_{1})  \gamma(a_{2}) \iota(a_{1}^{\gamma(a_{2})}) \iota(a_{2})
    \\&=
    \gamma(a_{1}^{\gamma(a_{2})} a_{2}) \iota(a_{1}^{\gamma(a_{2})} a_{2}),
  \end{align*}
  with $a_{1}^{\gamma(a_{2})} a_{2} \in A$, as $A$ is $\gamma(A)$-invariant.

  Conversely,   let  $\gamma'   :  A\to   \Aut(G)$  be   a  RGF   such
  that~\eqref{item:cap}~and   \eqref{item:g-i}    hold,   and   define
  $\gamma(ab)=\gamma'(a)$  for  each  $a\in  A$,  $b\in  B$.  The  map
  $\gamma$ is well-defined; in fact for  $i=1,2$, let $a_{i} \in A$ and
  $b_{i} \in B$ be such that  $a_{1} b_{1} = a_{2} b_{2}$; then $a_{1}
  = a_{2}  b_{2} b_{1}^{-1}$, so $b  = b_{2} b_{1}^{-1} \in  A \cap B$
  and
  \begin{equation*}
    \gamma'(a_{1}) 
    = 
    \gamma'( a_{2} b ) 
    = 
    \gamma'(a_{2}^{\gamma'(b)^{-1}}) \gamma'(b) 
    = 
    \gamma'(a_{2}).
  \end{equation*}
  Moreover
  \begin{align*}
    \gamma(a_{1} b_{1}) \gamma(a_{2} b_{2})
    =
    \gamma'(a_{1}) \gamma'(a_{2})
    =
    \gamma'(a_{1}^{\gamma(a_{2})} a_{2})
  \end{align*}
  and 
  \begin{align*}
    \gamma((a_{1} b_{1})^{\gamma(a_{2} b_{2})} a_{2} b_{2})
    &=
    \gamma(a_{1}^{\gamma(a_{2})} b_{1}^{\gamma(a_{2})} a_{2} b_{2})
    \\&=
    \gamma(a_{1}^{\gamma(a_{2})} a_{2} 
    b_{1}^{\gamma(a_{2}) \iota(a_{2})} b_{2})
    \\&=
    \gamma'(a_{1}^{\gamma(a_{2})} a_{2}),
  \end{align*}
  where the last  equality holds because of $(2)$; thus  $\gamma$ is a
  GF  on  $G$. Finally,  $\gamma(ab)=\gamma'(a) = 1$  if  and only  if
  $a \in \ker(\gamma')$, so $\ker(\gamma) = \ker(\gamma')B$.
\end{proof}

\begin{corollary}
  \label{cor:mor}
  In the notation of Proposition~\ref{prop:1.4} , let $\gamma$ be the
  lifting of $\gamma'$ to $G$. Then $\gamma$ is a morphism if and only
  if $\gamma'$ is  a morphism and $\ker(\gamma)$ is  a normal subgroup
  of $G$.
\end{corollary}

\begin{proof}
  Clearly,  if $\gamma$  is  a  morphism then  so  is its  restriction
  $\gamma'$, and $\ker(\gamma)$ is a normal subgroup of $G$.

  Conversely,  if $\gamma'$  is  a morphism  and  $\ker(\gamma)$ is  a
  normal subgroup of $G$ then for $a_i\in A$ and $b_i\in B$, $i=1,2$, we have
  \begin{equation*}
    \gamma(a_1b_1a_2b_2)
    =
    \gamma(a_1a_2b_1^{a_2}b_2)
    =
    \gamma'(a_1a_2)
    =
    \gamma'(a_1)\gamma'(a_2)
    =
    \gamma(a_1b_1)\gamma(a_2b_2).
  \end{equation*}
\end{proof}

We now aim at establishing a criterion
(Proposition~\ref{prop:images_gamma}) that allows us to define a map
$\gamma' : A \to \Aut(G)$  which verifies the  GFE, in the  case when
$A$ is a cyclic $p$-group.

First, we state separately two elementary arithmetic lemmas which will
be useful in the following. The first one is well-known.
\begin{lemma}
  \label{lemma:arithmetic}
  Let $p > 2$  be a prime and let $n>m\ge0$ be  integers. The solutions of
  the congruence
  \begin{equation}
    \label{cong}
    x^{p^m}\equiv1\pmod {p^n}
  \end{equation}
  are the integers of type $x=1+hp^{n-m}$.
\end{lemma}

\begin{lemma}
  \label{lemma:arithmetic2}
  Let $p > 2$ be a prime and 
  let $s\in\Z$, $s\equiv1\pmod{p}$. Define $\s(0)=0$ and for each $k>0$
  \begin{equation*}
    \s(k) = \sum_{i=0}^{k-1}s^i.
  \end{equation*}
  Then,  for each  $n\in \N$,  the set
  $\Set{\s(0),\dots,\s(p^n-1)}$  is  a  set of  representatives  of  the
  classes modulo $p^n$.\\
\end{lemma}

\begin{proof}
  We will show that for $k,h\in\Z$
  \begin{equation*}
    \s(k) \equiv \s(h) \pmod{p^n}
    \iff
    k \equiv h \pmod{p^n}
  \end{equation*}
  and this  implies the lemma.
  Let 
  \begin{equation}
    \label{eq:s}
    s=1+p^{l} h\quad \text{with  $\gcd(h,p)=1$};
  \end{equation}
  by  hypothesis $l>0$.  If  $k>h\ge0$, taking  into account  equality
  \eqref{eq:s},  we   have
  \begin{align*}  
    \s(k)\equiv  \s(h)\pmod{p^n}
    &\iff        
    s^h\sum_{i=0}^{k-h-1}s^i\equiv0\pmod{p^n}\\       
    &\iff
    \sum_{i=0}^{k-h-1}s^i\equiv0\pmod{p^n}\\                       
    &\iff
    (s-1)\sum_{i=0}^{k-h-1}s^i   \equiv   0  \pmod{p^{n+l}}\\   
    &\iff
    s^{k-h}\equiv 1 \pmod{p^{n+l}}.
  \end{align*}
  By Lemma~\ref{lemma:arithmetic}, $s=1+p^{l} h$ ($\gcd(h,p)=1$) is a
  solution of the last equation if and only if $k-h\equiv 0\pmod
  {p^n}.$ 
\end{proof}

\begin{corollary}
  \label{cor:same_generator}
  Let $G$ be a finite group, and let $A = \Span{a}$ be a cyclic
  subgroup  of $G$ of order $p^{n}$, where $p$ is an odd prime. Let
  \begin{equation*}
    \gamma : A \to \Aut(G)
  \end{equation*}
  be a RGF and let $a^{\gamma(a)}=a^s$. 
  
  Then,   for each $k$,
  \begin{equation}
    \label{eq:same_order}
    a^{\circ k}=a^{\s(k)},
  \end{equation}
  so $\ord_{(A,\circ)}(a)=\ord_A(a)$ and $(A, \circ)$ is  generated by
  $a$. 
\end{corollary}

\begin{proof}
  The  equality  in  \eqref{eq:same_order}  can  be  easily  shown  by
  induction;    the     corollary    follows    then     from    Lemma
  \ref{lemma:arithmetic2} since $a^{\circ k}\in A$ for each $k$.
\end{proof}
    
\begin{proposition}
  \label{prop:images_gamma}
  Let $G$ be a finite group, and let $A = \Span{a}$ be a cyclic
  subgroup  of $G$ of order $p^{n}$, where $p$ is an odd prime. 

  Let $\eta \in \Aut(G)$.
  
  The following are equivalent.
  \begin{enumerate}
  \item\label{item:theresagammaprime} 
    There is a RGF
    \begin{equation*}
      \gamma: A \to \Aut(G)
    \end{equation*}
    such that $\gamma(a) = \eta$.
  \item\label{item:invariant-and-power}
    \begin{enumerate}
    \item 
      $A$ is $\eta$-invariant, and
    \item 
      $\ord(\eta) \mid p^{n}$.
    \end{enumerate}
  \end{enumerate}
  When these  conditions hold, $\gamma$ is uniquely defined.
\end{proposition}

\begin{proof}
  Assume first that the map
  $\gamma\colon   A\to\Aut(G)$ is a RGF,   and   let
$\gamma(a)=\eta$.  
Then,  $\gamma\colon (A, \circ)\to\Aut(G)$
  is a  morphism, so  
\begin{equation*}
\ord(\eta)~\mid \ord_{(A,\circ)}(a)=\ord_A(a)=p^n
\end{equation*}
where the first equality follows from Corollary~\ref{cor:same_generator}.   \smallskip

  As to  the converse,  assume~\eqref{item:invariant-and-power} holds.
  Then $\eta_{\restriction A} \in  \Aut(A) \cong (\Z/p^n\Z)^*$ and, if $\eta(a)=a^s$, then $s\in(\Z/p^n\Z)^*$ 
  and 
    \begin{equation*}
    \ord(s)=\ord(\eta_{\restriction A})    
    \mid     
    \gcd(p^n ,\phi(p^{n}))
    =
    p^{n-1},
  \end{equation*}
  so that,  
  by  Lemma~\ref{lemma:arithmetic},
  \begin{equation}
    \label{eq:ords}s\equiv 1\pmod p. 
  \end{equation} 

  In the notation of Lemma~\ref{lemma:arithmetic2} we have that $\Set{
    \s(0),  \dots,\s(p^n -1)  }$ is  a set  of representatives  of the
  classes modulo $p^n$, hence $A = \Set{ a^{\s(k)} }_{k=0}^{p^n-1}$.
  Therefore we can define $\gamma$ on $A$ letting, for all $k$,
  \begin{equation*}
    \gamma (a^{\s( k)})=\eta^k,
  \end{equation*}
  and we have only to check that it satisfies the GFE.
  Now $a_1=a^{\s( k_1)}$, $a_2=a^{\s( k_2)}$, for some $k_1$, $k_2$,
  so that
  \begin{align*}
    \gamma (a_1^{\gamma (a_2)}a_2)
    &=\gamma ({(a^{\s( k_1)})}^{\gamma (a^{\s( k_2)})}a^{\s( k_2)})\\
    &=\gamma ({(a^{\s( k_1)})}^{\eta^{k_2}}a^{\s( k_2)})\\
    &=\gamma (a^{s^{k_2}\sum_{i=0}^{k_1-1}s^i}a^{\sum_{i=0}^{k_2-1}s^i})\\
    &=\gamma (a^{\sum_{i=0}^{k_1+k_2-1}s^i})\\
    &=\gamma (a^{\s (k_1+k_2)})\\
    &=\eta^{k_1+k_2}\\
    &=\gamma (a^{\s(k_1)})\gamma (a^{\s(k_2)})\\
    &=\gamma ({a_1})\gamma ({a_2}).
  \end{align*}
  Finally,  if  $\gamma '\colon  A\to\Aut(G)$  is a RGF such that
  $\gamma'(a)=\eta$,  denoting by  $\circ'$ the  operation $a_1\circ'
  a_2=a_1^{\gamma'(a_2)}a_2$,          we          have          that
  $\gamma'(a^{\circ'k})=\eta^k$, for each $k$. On the other hand,
  \begin{equation*}
    a^{\circ'k}=a^{\sum_{i=0}^{k-1}\eta^i}=a^{\s( k)},
  \end{equation*}
  so that
  \begin{equation*}
    \gamma'(a^{\s(k)})=\gamma'(a^{\circ' k})=\eta^k=\gamma(a^{\s(k)}),
  \end{equation*}
  that is, $\gamma'=\gamma$.
\end{proof}

\begin{corollary}
  \label{cor:morA}
  Let $G$ be a finite group, and let $A = \Span{a}$ be a cyclic
  subgroup of $G$ of order $p^n$ where $p$ is an odd prime. 
  
  Let $\gamma\colon A\to\Aut(G)$ be a RGF. 
  
  Then the following are equivalent:
  \begin{enumerate}
  \item 
    $\gamma$ is a morphism, and
  \item 
    $a^{\gamma(a)}=a^s$, with $s \equiv 1 
  \pmod{\ord(\gamma(a))}$.
  \end{enumerate}
\end{corollary}

\begin{proof}
  Let $\gamma(a)=\eta$. Then    $\gamma$ is a morphism if  and only if,
  for all $k$,
  \begin{equation*}
    \gamma(a^{\s(k)})=\gamma(a)^{\s(k)},
  \end{equation*}
  or equivalently 
  \begin{equation*}
    \eta^k=\eta^{\s(k)},
  \end{equation*}
  namely, $\s(k)\equiv k\pmod{\ord(\eta)}$. 
  The
  last  condition  is  easily  seen   to  be  equivalent  to  $s\equiv
  1\pmod{\ord(\eta)}$.
  
  In fact,  if $s\equiv 1\pmod{\ord(\eta)}$ then  clearly $\s(k)\equiv
  k\pmod{\ord(\eta)}$ for all $k$.

  On  the the  hand, if  $\s(k)\equiv k\pmod{\ord(\eta)}$  for all  $k$,
  then, in  particular, $\s(2)=s^0+s=  2+(s-1)\equiv 2\pmod{\ord(\eta)}$
  namely $s\equiv 1\pmod{\ord(\eta)}$.
  \end{proof}

\begin{corollary}
  \label{cor:ordp}
  In    the    notation    of   Corollary~\ref{cor:morA},    assume
  \begin{equation*}
    \ord(\gamma(a)) = p.
  \end{equation*}
  Then $\gamma$ is a morphism.
\end{corollary}
\begin{proof}
  This follows from Corollary~\ref{cor:morA} and equation~\eqref{eq:ords}. 
\end{proof}

\subsection{Duality}
\label{ssec:duality}

The  GF  associated  to  the  image $\rho(G)$  of  the  right  regular
representation is $\gamma(G)  = \Set{ 1 }$, and  the associated circle
operation on $G$ is the defining operation on the group $G$.

The GF  associated to the  image $\lambda(G) = \rho(G)^{\inv}$  of the
left  regular representation  (see
Proposition~\ref{prop:right-and-left}\eqref{item:inversion}) is
$\iota(y^{-1})$,  and the  associated 
circle operation  is the  opposite operation,  $x \circ y  = y  x$, as
$x^{\iota(y^{-1})} y = y x  = x^{\lambda(y)}$.  In particular, 
\begin{align*}
  \inv :\ &G \to (G, \circ)\\
  &x \mapsto x^{-1}
\end{align*}
is an isomorphism in this case.

Our  next result  is an  extension of  the above  pairing between  the
images  of the  right  and  the left  regular  representations to  all
regular subgroups of  $\Hol(G)$. This will be useful, as  it allows us
to halve the number of GF we have to consider, when $G$ is non-abelian,
and also  because it allows  us in some  circumstances to choose  a GF
with  a   kernel  that   is  more   suitable  for   calculations  (see
Proposition~\ref{prop:duality} below).
\begin{proposition}
  \label{prop:gammatilde}
  Let $G$ be  a finite group, $\gamma  : G \to \Aut(G)$ a  GF, $N$ the
  associated regular subgroup of $\Hol(G)$, and $\circ$ the associated
  operation. 

  Then
  \begin{align*}
    \tilde\gamma :\ &G \to \Aut(G)\\
    &x \mapsto \gamma(x^{-1}) \iota(x^{-1})
  \end{align*}
  is also a GF, which corresponds to the regular subgroup
  $N^{\inv}$, that is, the conjugate of $N$ under $\inv \in S(G)$. If
  $\tilde{\circ}$ is the operation associated to 
  $\tilde\gamma$, then 
  \begin{equation*}
    \inv : (G, \circ) \to (G, \tilde\circ)
  \end{equation*}
  is an isomorphism.
\end{proposition}

\begin{proof}
  From Proposition~\ref{prop:right-and-left}\eqref{item:inversion}
  we have that $\inv$ normalises $\Hol(G)$.

  Consider the conjugate of  $N = \Set{ \gamma(y)  \rho(y) : y
    \in G }$  under $\inv$, which will be another  regular subgroup of
  $\Hol(G)$. We have
  \begin{align*}
    x^{(\gamma(y)  \rho(y))^{\inv}} 
    &=  
    (x^{-\gamma(y)} y)^{-1}  
    \\&=
    y^{-1} x^{\gamma(y)} 
    \\&= 
    x^{\gamma(y) \iota(y) \rho(y^{-1})}.
  \end{align*}
  In particular,  $1^{(\gamma(y) \rho(y))^{\inv}} =  y^{-1}$,
  that is, $(\gamma(y) \rho(y))^{\inv}$
  is the element  of $N^{\inv}$ taking $1$ to  $y^{-1}$. Therefore the
  GF associated to $N^{\inv}$ is
  \begin{align*}
    \tilde\gamma  :\  &G  \to   \Aut(G)\\  &y  \mapsto  \gamma(y^{-1})
    \iota(y^{-1}).
  \end{align*}
  The operation associated to $\tilde\gamma$ is
  \begin{multline*}
    x  \mathbin{\tilde{\circ}} y  
    =  
    x^{\tilde\gamma(y)}  y 
    =  
    x^{\gamma(y^{-1})   \iota(y^{-1})} y  
    =\\=
    y x^{\gamma(y^{-1})}  
    = 
    (x^{-\gamma(y^{-1})} y^{-1})^{-1} 
    = 
    (x^{-1} \circ y^{-1})^{-1},
  \end{multline*}
  so that, as in  the case of $G$ and its opposite  group,
  $\inv : (G,  \circ) \to (G, \tilde \circ)$ is an isomorphism. (See
  also~\cite[Lemma 1.4]{perfect}.)
\end{proof}

\begin{lemma}
  \label{lemma:duality}
  Let $G$ be a finite non-abelian group.
  Let $C$ be a non-trivial subgroup of $G$ such that:
  \begin{enumerate}
  \item\label{item:C-is-abelian}
    $C$ is abelian;
  \item\label{item:C-is-char}
    $C$ is characteristic in $G$;
  \item $C\cap Z(G)=\{1\}.$
  \end{enumerate}
  Let $\gamma\colon G\to\Aut(G)$  be a GF, and suppose  that for every
  $c\in C$  we have  $\gamma(c)=\iota(c^{-\sigma})$ for  some function
  $\sigma\colon C\to C$.  

  Then  $\sigma\in\End(C)$, and for every $a\in
  G$ the following relation holds in $\End(C)$:
  \begin{equation}
    \label{eq:sigma}
    \sigma \, \gamma(a)_{\restriction C} \, (\sigma - 1)
    =
    (\sigma - 1) \, \gamma(a)_{\restriction C} \, \iota(a)_{\restriction
      C} \, \sigma
  \end{equation}
\end{lemma}

\begin{proof}
  For $c_1, c_2\in C$ we have
  \begin{align*}
    \iota((c_1c_2)^{-\sigma})&=\gamma(c_1c_2)\\
    &=\gamma(c_1^{\gamma(c_2)^{-1}})\gamma(c_2)\\
    &=\gamma(c_1^{\iota(c_2^{\sigma})})\gamma(c_2)\\
    &=\iota(c_1^{-\sigma})\iota(c_2^{-\sigma})\\
    &=\iota( c_1^{-\sigma}c_2^{-\sigma}).
  \end{align*}
  Since  $C$  is  abelian,  and $C\cap  Z(G) = \Set{1}$,  we  obtain  that
  $\sigma$ is an endomorphism of $C$.

  For $a \in G$ and $c \in C$ we have
  \begin{equation}
    \label{eq:defining_tau}
    \begin{aligned}
      a^{\ominus 1} \circ c \circ a
      &=
      a^{-\gamma(a)^{-1} \gamma(c) \gamma(a)} c^{\gamma(a)} a
      \\&=
      a^{\iota(c^{- \sigma \gamma(a)})} a c^{\gamma(a) \iota(a)}
      \\&=
      c^{\sigma \gamma(a)} a^{-1} c^{- \sigma \gamma(a)} a c^{\gamma(a)
        \iota(a)}
      \\&=
      c^{\sigma \gamma(a) - \sigma \gamma(a) \iota(a) + \gamma(a)
        \iota(a)}.
    \end{aligned}
  \end{equation}
  On the other hand
  \begin{equation*}
    \gamma(a^{\ominus 1} \circ c \circ a)
    =
    \gamma(a)^{-1} \gamma(c) \gamma(a)
    =
    \iota(c^{- \sigma \gamma(a)}),
  \end{equation*}
  so  that, writing  $\tau$ for  the exponent  $\sigma  \gamma(a) -
  \sigma \gamma(a) \iota(a)  + \gamma(a) \iota(a)$ of $c$  in the last
  term of~\eqref{eq:defining_tau}, we get
  \begin{equation*}
    \gamma(c^{\tau})
    =
    \iota(c^{- \tau \sigma})
  \end{equation*}
  Since $C \cap Z(G) = \Set{1}$, we obtain $\sigma \gamma(a) =
  \tau \sigma$, and thus~\eqref{eq:sigma}
\end{proof}

\begin{proposition}
  \label{prop:duality}
  Let $G$ be a finite non-abelian group.
  Let $C$ be a subgroup of $G$ such that:
  \begin{enumerate}
  \item 
    $C = \Span{c}$ is cyclic, of order a power of the prime $r$,
  \item
    $C$ is characteristic in $G$,
  \item
    $C \cap Z(G) = \Set{1}$, and
  \item 
    there  is  $a \in  G$  which  induces  by  conjugation on  $C$  an
    automorphism whose order is not a power of $r$.
  \end{enumerate}
  Let $\gamma\colon G\to\Aut(G)$  be a GF, and suppose  that for every
  $c\in C$  we have  $\gamma(c) = \iota(c^{-\sigma})$, for  some function
  $\sigma \colon C \to C$.  

  Then
  \begin{enumerate}
  \item
    either $\sigma = 0$, that is, $C \le \ker(\gamma)$, 
  \item
    or $\sigma = 1$, that is, $\gamma(c) = \iota(c^{-1})$, so that $C
    \le \ker(\tilde\gamma)$.
  \end{enumerate}
\end{proposition}

Note that the hypotheses of Proposition~\ref{prop:duality} contain
those of Lemma~\ref{lemma:duality}.

\begin{proof}
  It is immediate that $\sigma \in \End(C)$, so that we can identify
  $\sigma$ with an integer modulo the order of $c$.
  
  Since  $\End(C)$ is  abelian,  and  $\gamma(a)_{\restriction C}  \in
  \Aut(C)$, we obtain from~\eqref{eq:sigma} the equality
  \begin{equation*}
    \sigma (\sigma - 1) (\iota(a)_{\restriction C} - 1) = 0
  \end{equation*}
  in $\End(C)$, for all $a \in G$.  Choose now $a \in G$ which induces
  on $C$  an automorphism  $\iota(a)$ whose  order is  not a  power of
  $r$. Then $\iota(a)_{\restriction C} - 1$ is not a zero divisor in
  $\End(C)$, so that 
  \begin{equation*}
    \sigma (\sigma - 1) = 0.
  \end{equation*}
  Since $\End(C)$ is a local ring, we obtain
  that either $\sigma = 0$,
  or $\sigma = 1$.
  
  If the latter holds we have then
  \begin{equation*}
    \tilde\gamma(c) 
    = 
    \gamma(c^{-1}) \iota(c^{-1}) 
    =
    \iota(c) \iota(c^{-1})
    = 
    1.
  \end{equation*}
\end{proof}

\begin{corollary}
  \label{cor:duality}
  Let $G$  be a  finite non-abelian  group, $r$  a prime  dividing the
  order of $G$, and $B$ a Sylow $r$-subgroup of $G$.

  Suppose that
  \begin{itemize}
  \item $B$ is cyclic, and
  \item $B$ contains a subgroup $C$  of order $r$, which satisfies the
    hypotheses of 
    Proposition~\ref{prop:duality}. 
  \end{itemize}
  
  Let
  $\gamma$ be  a GF on  $G$. For each group  $\mathcal G$ of  the same
  order as $G$, let
  \begin{equation*}
    k_r(\mathcal G)
    =
    \Size{
      \Set{
        \text{$\gamma$ GF on $G$} :
        r \mid  \Size{\ker(\gamma)} \text{  and  }  
        (G, \circ)  \cong  \mathcal  G
      }
    }.
  \end{equation*}
  Then
  \begin{equation*}
    e'(\mathcal  G,  G)
    =
    \Size{
      \Set{
        \text{$\gamma$ GF on $G$}
        :
        (G, \circ) \cong \mathcal G
      }
    }
    =
    2k_r (\mathcal G).
  \end{equation*}
\end{corollary}

\begin{proof} 
  Write
  \begin{align*}
    &X = \Set{
      \text{$\gamma$ GF on $G$}
      :
      (G, \circ) \cong \mathcal G
    }, \\
    &X_1 = \Set{
      \text{$\gamma$ GF on $G$}
      :
      \text{$(G, \circ) \cong \mathcal G$ and $r \mid \Size{\ker(\gamma)}$} 
    }, \\
    &X_2 = \Set{ 
      \text{$\gamma$ GF on $G$}
      :
      \text{$(G, \circ) \cong \mathcal G$ and $r \nmid \Size{\ker(\gamma)}$}
    }.
  \end{align*}
  We have
  \begin{equation*} 
    e'(\mathcal G, G)=|X|=|X_1|+|X_2|=k_r(\mathcal G)+|X_2|.
  \end{equation*}
  If we  show that there  is a bijection  between $X_1$ and  $X_2$, it
  will follow that $e'(\mathcal G, G)=2k_r(\mathcal G)$. Consider
  \begin{align*}
    \psi \colon X &\to X \\
    \gamma &\mapsto \tilde\gamma ,
  \end{align*}
  where $\tilde\gamma$ is as in Proposition~\ref{prop:gammatilde}. The
  map $\psi$ is well defined, indeed $\tilde\gamma$ is a GF on $G$ and
  $(G,\tilde\circ)\cong (G,\circ) \cong \mathcal  G$ (see the proof of
  Proposition~\ref{prop:gammatilde}); moreover
  \begin{equation*}
    \psi^2(\gamma)
    = 
    \psi(\tilde\gamma)
    =
    \tilde{\tilde{\gamma}}.
  \end{equation*}
  By Theorem~\ref{thm:gamma-for-regular} each GF on $G$ corresponds to
  a unique  regular subgroup of $\Hol(G)$,  so let $N$ be  the regular
  subgroup   corresponding   to   $\gamma$;   then,   by
  Proposition~\ref{prop:gammatilde},   the  regular   subgroup
  corresponding   to 
  $\tilde{\tilde{\gamma}}$ is  ${(N^{\inv})}^{\inv} =  N$, so  that we
  get  that   $\psi^2=1$,  and  $\psi$  is   bijective. Now, using
  Proposition~\ref{prop:duality},  we  get that  $\psi(X_2) = X_1$,
  so that
  $\Size{X_2} = \Size{X_1}$.
\end{proof}

Note that this duality is equivalent to the notion of an
\emph{opposite skew brace} as introduced by Koch and Truman
in~\cite{opposite}. In fact, given a brace $(G, \cdot, \circ)$, Koch
and Truman define the opposite brace to be $(G, \cdot', \circ)$, where
$x \cdot' y = y x$ gives the opposite group $(G, \cdot')$ of $(G,
\cdot)$. With our construction, the circle operation associated to the
regular subgroup $N^{\inv}$ is given by $x \mathbin{\tilde\circ} y = x^{\tilde
  \gamma(y)} \cdot y = x^{\gamma(y^{-1}) \iota(y^{-1})} \cdot y = y \cdot
x^{\gamma(y^{-1})}$.

Now $\inv : (G, \cdot', \circ) \to
(G, \cdot, \tilde\circ)$ is an isomorphism of skew braces, as for $x, y \in
G$ we have $(x \cdot' y)^{\inv} = (y \cdot x)^{-1} = x^{-1} \cdot
y^{-1} = x^{\inv} \cdot y^{\inv}$, and, $(x \circ y)^{\inv} =
(x^{\gamma(y)} \cdot y)^{-1} = y^{-1} \cdot x^{-\gamma(y)} = x^{-\gamma(y)
  \iota(y)} \cdot y^{-1} = x^{\inv}
\mathbin{\tilde\circ} y^{\inv}$.

\subsection{An application: groups of order $p q$}
\label{subs:pq}

To exemplify our methods, we first apply  them to the case, dealt with by
Byott in~\cite{Byott04}, of  groups of order $p q$, where  $p$ and $q$
are distinct primes. We also recover the classification of skew
braces of order $p q$ of Acri and Bonatto~\cite{AcriBonatto_pq}.

So let $p > q$ be two primes. We will write $\cC_{p q}$ for the cyclic
group of order $p q$, and $\cC_{p} \rtimes \cC_{q}$ for the non-abelian
one, which occurs when $q \mid p - 1$.

Byott has proved the following

\begin{theo}[\protect{\cite[Section 6]{Byott04}}]
Let $L/K$ be a Galois field  extension of
order $p q$, and let $\Gamma = \Gal(L/K)$.

Then the following
table gives  the numbers  $e(\Gamma, G)$ of  Hopf-Galois  structures on
$L/K$ of type $G$ for each group $G$ of order $p q$.
\begin{center}
  \begin{tabular}{r|cc}
    \diagbox[width=4.0em]{$\Gamma$}{$G$} & $\cC_{p q}$ & $\cC_p\rtimes\cC_q$ \\
    \hline
    $\cC_{p q}$ & $1$ & $2(q-1)$ \\
    $\cC_p\rtimes\cC_q$ & $p$ & $2(p q-2p+1)$ \\
  \end{tabular}
\end{center}
\end{theo}
We now compute with our methods the number  $e'(\Gamma, G)$ of the
regular subgroups of
$\Hol(G)$ which are isomorphic to $\Gamma$, in the form
\begin{center}
  \begin{tabular}{r|cc}
    \diagbox[width=4.0em]{$\Gamma$}{$G$} & $\cC_{p q}$ & $\cC_p\rtimes\cC_q$ \\
    \hline
    $\cC_{p q}$ & $1$ & $2p$\\
    $\cC_p\rtimes\cC_q$ &  $q - 1$  & $2(p q-2p+1)$ \\
  \end{tabular}
\end{center}
from which the previous theorem can be obtained using
formula~\eqref{eq:e_vs_e-prime} of Theorem~\ref{th:byott96}.

In terms
of conjugacy classes of regular subgroups, we have
\begin{theorem}
  Let $G = (G, \cdot)$ be a group of order $p q$, where $p, q$ are
  primes, with $p > q$.

  For each group $\Gamma$ of order $p q$, the following table gives
  equivalently 
  \begin{enumerate}
  \item
    the number (and lengths) of conjugacy
    classes within $\Hol(G)$  of regular 
    subgroups isomorphic to $\Gamma$;
  \item
    the number of isomorphism classes of braces $(G, \cdot, \circ)$
    such that $\Gamma \cong (G, \circ)$.
  \end{enumerate}
  \begin{center}
    \begin{tabular}{r|cc}
      \diagbox[width=4.0em]{$\Gamma$}{$G$} & $\cC_{p q}$ & $\cC_p\rtimes\cC_q$ \\
      \hline
      $\cC_{p q}$ & $(1, 1)$ & $(2, p)$ \\
      $\cC_p\rtimes\cC_q$ & $(1, q-1)$ & $(2, 1), (2 (q - 2), p)$ \\
    \end{tabular}
  \end{center}
  Here $(c, l)$ denotes $c$ conjugacy classes of length $l$; the full
  table refers to the case when $q \mid p - 1$, and the $1 \times 1$
  upper left sub-table refers to the case $q \nmid p - 1$.
\end{theorem}
We obtain that when $q \mid p - 1$ there are $2 q + 2$ isomorphism
classes of braces of order $p q$, which coincides with the results
of~\cite{AcriBonatto_pq}. 

Let $\gamma$ be a GF on $G$,  let $B$ be the Sylow $p$-subgroup of $G$
and $A$ a Sylow $q$-subgroup.

If $G = \cC_{p q}$, then $B \le \ker(\gamma)$,
since  in  $\Aut(G) \cong  \cC_{p-1}\times\cC_{q-1}$  there  are  no
elements of order  $p$.  

If  $G =  \cC_p\rtimes\cC_q$, we may take  $r  = p$  and  $C =  B$ in  the
hypotheses  of Corollary~\ref{cor:duality} here, so that we need only to
consider the case
$B \le \ker(\gamma)$.

If $\ker(\gamma) = G$, we get the right regular representation, whose
image forms a conjugacy class in itself.

Suppose thus  $\ker(\gamma) = B$,  so that $\Size{\gamma(G)} =  q$. We
claim that there is  a unique $\gamma(G)$-invariant Sylow $q$-subgroup
$A$  of $G$.  This  is clearly  true for  $G =  \cC_{p q}$.  When  $G =
\cC_p\rtimes\cC_q$, the $q$-elements of $\Aut(G) \cong \cC_{p} \rtimes
\cC_{p-1}$  are  inner  automorphisms,  so that  $\gamma(G)  =  \Span{
  \iota(a) }$ for  some $a \in G$  of order $q$. It follows  that 
$A = \Span{a}$ is the unique $\gamma(G)$-invariant Sylow $q$-subgroup.

Now $A \gamma(G)$ is a subgroup of order $q^{2}$ of $\Hol(G)$, so that
$[A, \gamma(G)] = 1$, and thus 
\begin{equation*}
  [G, \gamma(G)] 
  =
  [A B, \gamma(G)]
  =
  [B, \gamma(G)]
  =
  [B, \Span{\iota(a)}]
  =
  B 
  = 
  \ker(\gamma).
\end{equation*}
By Lemma~\ref{Lemma:gamma_morfismi},  all such GF's are  precisely the
morphisms $\gamma : G \to \Aut(G)$ of groups with kernel $B$.

If $G = \cC_{p q}$,  and $q \nmid p - 1$, there  are no such morphisms,
as in this case $q \nmid \Size{\Aut(G)}$. If $q \mid p - 1$, there are
exactly  $q-1$  such morphisms  with  kernel  of  order $p$,  as  they
correspond to  sending a fixed element of  order $q$ of $G$  to one of
the $q-1$ elements of order $q$ of $\Aut(G)$. The corresponding
regular subgroups are non-abelian by
Lemma~\ref{lemma:inverse_and_conjugacy}, as for $b \in B$ we have
\begin{equation*}
  a^{\ominus 1} \circ b \circ a
  =
  b^{\gamma(a)}
  \ne 
  b.
\end{equation*}
Let $\beta \in \Aut(G)$ send $a$ to $a^{t}$, for some $t$. Then,
according to Lemma~\ref{lemma:conjugacy}, and since $\Aut(G)$ is
abelian, we have $\gamma^{\beta}(a) = \gamma(a^{\beta^{-1}}) =
\gamma(a)^{t^{-1}}$. It follows that all these $q - 1$ GF's are conjugate.

If $G = \cC_p \rtimes \cC_q$, one has  first to choose a $\Span{\iota(a)}$
among the $p$ Sylow $q$-subgroups of $\Aut(G)$. Since for $b \in B$
Lemma~\ref{lemma:inverse_and_conjugacy} yields
\begin{equation*}
  a^{\ominus 1} \circ b \circ a
  =
  a^{-1} b^{\gamma(a)} a
  =
  b^{\gamma(a) \iota(a)},
\end{equation*}
the choice  $\gamma(a) = \iota(a)^{-1}$  yields $p$ instances  of $(G,
\circ)          =           \cC_{p q}$.           According          to
Lemma~\ref{lemma:conjugacy}.\eqref{item:invariance-under-conjugacy},
all these  $\gamma$ are  conjugate under $\iota(B)$,  which conjugates
transitively the Sylow $q$-subgroups.

The other non-trivial  choices of $\gamma(a) =  \iota(a)^{s}$, with $s
\ne 0, -1$ yield  $p (q - 2)$ instances of  $(G, \circ) = \cC_p \rtimes
\cC_q$.   Once  more,  the  action  of  $\iota(B)$,  which  conjugates
transitively the Sylow $q$-subgroups, shows that the conjugacy classes
are of length at least $p$.  The cyclic complement of order $p - 1$ of
$\iota(B)$ in $\Aut(G)$ which contains $\gamma(a) = \iota(a^{s})$ then
centralises $a$  and $\gamma(a)$, so  that it centralises  $\gamma$ by
Lemma~\ref{lemma:conjugacy}, and the conjugacy classes have length
precisely $p$.

\subsection{An application: a result of Kohl}
\label{subs:Kohl}

In~\cite{Ko13, Ko16}, Kohl gives a method to determine the Hopf-Galois
structures on a family of Galois extensions of degree $pm$, where $p$
is a prime, 
$\gcd(m, p)=1$, and some additional hypotheses hold.

With our methods we obtain the following slight generalisation
of  the main  result~\cite[Theorem 1.3]{Ko16},  which we  have
reformulated in terms of Byott's translation. 

\begin{theorem}
  Let $G$ be a group of order $m p$, with $p \nmid m$. 
  Assume $G$ has a unique Sylow $p$-subgroup $P$. 

  Let $M \le G$  be a subgroup of order $m$. Assume  $p$ does not divide
  $\Size{\Aut(M)}$.

  Let $N$ be a regular subgroup of $\Hol(G)$. 

  Then
  \begin{enumerate} 
  \item
    $\nu(P) \le  N$, so  that $\nu(P)  $ is a  Sylow $p$-subgroup  of $N$,
    which need not be unique.
  \item
    $\nu(P) \in \Set{\rho(P), \lambda(P) } $. 
  \item
    $[\rho(G), \nu(P)] \le \nu(P) $. 
  \end{enumerate} 
\end{theorem} 

Note that the  subgroup $M$ as in the statement  of the Theorem exists
by the theorem of Schur-Zassenhaus.

\begin{proof} 
  Since $P = \Span{x}$ is characteristic in $G$, it is also a subgroup
  of $(G, \circ) $,  so that $\nu(P) \le N$, as $\nu  : (G, \circ) \to
  N$ is a morphism.

  Proposition~\ref{prop:images_gamma} yields  that $\gamma(x) $  is an
  element of $\Aut(G) $ of order dividing $p$.

  Suppose first $[P, M] = 1$. Then  there are no elements of order $p$
  in $\Aut(G) = \Aut(P) \times \Aut(M)$, so that $\gamma(x) = 1$. It
  follows  that   $\nu(P)  =   \rho(P)  $  in   this  case,   so  that
  $[\rho(G), \nu(P) ] = [\rho(G), \rho(P)] = \rho([G, P]) = 1$.

  Let thus  $[P, M] \ne  1$, so that $\iota(P)  $, the group  of inner
  automorphisms of $G$ induced by  conjugation by the elements of $P$,
  has order $p$.

  We  now go  through  the arguments  of~\cite[Lemma 1.1]{Ko13}  ~and
  \cite[Lemma 1.2]{Ko16}. According to~\cite[Theorem 1]{Curran} since
  $p$  does  not  divide  $\Size{\Aut(P)}$  and  $\Size{\Aut(M)  }  $,  the
  $p$-elements of $\Aut(G) $ are in one-to-one correspondence with the
  maps $\beta : M \to P$ such that
  \begin{equation} 
    \label{eq:beta2} 
    \beta(y z) = \beta(z) \beta(y)^{z}.
  \end{equation} 
  (Among these  maps, one  finds the maps  corresponding to  the inner
  automorphisms $\iota(P) $, in the form $\beta(y) = [x^{i}, y] $, for
  $i  =  0,  1,  \dots,  p-1$.)   According  to~\eqref{eq:beta2},  the
  restriction of $\beta$  to the centraliser $C_{M}(P)$ of  $P$ in $M$
  yields a morphism $C_{M} (P) \to P$; since $\gcd(\Size{M}, \Size{P})
  = 1$, we obtain  that $\beta$ is trivial on $C_{M}  (P)$. Since $M /
  C_{M}  (P)$  is  isomorphic  to  a  subgroup  of  the  cyclic  group
  $\Aut(P)$,  \eqref{eq:beta2}   shows  that  $\beta$   is  completely
  determined by its  value $\beta(g) $ on an element  $g$, whose image
  generates  $M /  C_{M}  (P)$.  There are  at  most  $p$ choices  for
  $\beta(g) \in P$. It follows that  the automorphisms of $G$ of order
  $p$ are inner, induced by conjugation by elements of $P$.

  Proposition~\ref{prop:duality}   then  yields   that
  $\gamma(x)   =  \iota(x^{-\sigma}) $, where
  \begin{itemize} 
  \item
    either $\sigma =  0$, so that $\nu(P) = \rho(P)  $, and thus
    \begin{equation*}
      [\rho(G),  \nu(P) ]
      =
      [\rho(G), \rho(P)]
      =
      \rho([G, P])
      =
      \rho(P)
      =
      \nu(P),
    \end{equation*}
  \item
    or $\sigma = 1$, and thus $\nu(P) = \lambda(P) $ is centralised by
    $\rho(G) $.  
  \end{itemize} 
\end{proof} 

\section{The groups of order $p^2q$ and their automorphism groups}
\label{sec:the-groups}

In this  section we  prepare the field for the proof of Theorem~\ref{number},
which will take place in the next section. Although in this paper we
deal  with those groups of order $p^{2}
q$, where $p$ and $q$ are distinct primes, which have cyclic Sylow
$p$-subgroups, some of our results will be stated for the general case.

We will first recall the classification of the groups of order $p^{2}
q$ with cyclic Sylow $p$-subgroups, where $p$ and $q$ are distinct
primes, and of their automorphism 
groups. We will then show  how to  apply the  duality results of
Subsection~\ref{ssec:duality} to the non-abelian groups of order
$p^{2} q$ with a cyclic Sylow $p$-subgroup. Finally we will show that
for odd $p$, if $G$ is a group of order $p^{2} q$, a Sylow $p$-subgroup of
a regular subgroup of $\Hol(G)$ is isomorphic to a Sylow $p$-subgroup
of $G$.

The classification  of groups  of order  $p^{2} q$,  where $p,  q$ are
distinct  primes,  goes   back  to  O.~H\"older~\cite{Holder1893}.  In
particular, H\"older  showed that in  such a  group there is  always a
normal Sylow  subgroup.  As a  handy reference, we have  recorded the
classification  of  these  groups  and of  their  automorphism  groups
in~\cite{classp2q}. We list in the table below those groups of order
$p^2q$ with a cyclic Sylow $p$-subgroup, referring to~\cite{classp2q}
for the details. Here each type corresponds to an isomorphism class. 
We use the notation $\cC_{n}$ for a cyclic
  group of order $n$.

\begin{description}
\item[Type 1] Cyclic group.
 \item[Type 2] This  is the  non-abelian group  with centre  of order
  $p$  for  $p \mid  q  -  1$,  which  we denote  by  $\cC_{p^{2}}
  \ltimes_{p} \cC_{q}$.
\item[Type 3] This is the non-abelian group with trivial centre for
  $p^{2} \mid q - 1$, which we denote by $\cC_{p^{2}} \ltimes_{1}
  \cC_{q}$.
\item[Type 4] This is the non-abelian group for $q \mid p - 1$, which
  we denote by $\cC_{p^{2}} \rtimes \cC_{q}$.
\end{description}

\begin{center}
  \begin{table}[ht!]
    \caption{Groups of order $p^2q$ with cyclic Sylow $p$-subgroups
      and their automorphisms} 
    \label{table:grp_aut}
    \begin{tabular}{c|c|c|c}
      Type & Conditions & $G$ & $\Aut(G)$ \\
      \hline\hline
      1 & & $\cC_{p^{2}} \times \cC_{q}$ & $\cC_{p (p - 1)} 
      \times \cC_{q-1}$
      \\\hline
      2 & $p \mid q - 1$ & $\cC_{p^{2}} \ltimes_{p} \cC_{q}$ &
      $\cC_{p} \times \Hol(\cC_{q})$ 
      \\\hline
      3 & $p^{2} \mid q - 1$ & $\cC_{p^{2}} \ltimes_{1} \cC_{q}$ &
      $\Hol(\cC_{q})$
      \\\hline 
      4 & $q \mid p - 1$ & $\cC_{p^{2}} \rtimes \cC_{q}$  &
      $\Hol(\cC_{p^{2}})$
      \\
    \end{tabular}
  \end{table}
\end{center}

The next Propositions apply to the groups in
Table~\ref{table:grp_aut}, and also to the other groups of order
$p^{2} q$. For the types of such groups not dealt with in this paper,
we refer to the notation of~\cite{classp2q}.
\begin{proposition}
  \label{prop:duality_p2q}
  Let $G$ be a non-abelian group of  order $p^2q$, and let $\gamma$ be a
  GF on  $G$. 
  \begin{enumerate}
  \item 
    Let $B$ be the  normal Sylow $r$-subgroup  of $G$
    ($r \in \Set{p, q}$);
  \item if $r=p$, assume $B$ cyclic;
  \item denote by $C$ the unique subgroup of $B$ of order $r$.
  \end{enumerate}
  Then
  \begin{equation*}
    C \le \ker(\gamma)   
    \text{   if  and   only   if   }   
    C \nleq \ker(\tilde\gamma).
  \end{equation*}
  Moreover,  for each group $\mathcal  G$ of order
  $p^2q$, let
  \begin{equation*}
  k_r(\mathcal G)
  = 
  \Size{
    \Set{ 
      \gamma \text{ GF on $G$}
      : 
      C\le \ker(\gamma) \text{ and } (G, \circ) \cong \mathcal{G} 
    }
  }.
  \end{equation*}
  Then
  \begin{equation*}
    e'(\mathcal G, G)
    =
    \Size{
      \Set{
        \gamma\text{ GF on $G$}
        :
        (G, \circ) \cong \mathcal G
      }
    }
    =
    2 k_r(\mathcal{G}).
  \end{equation*}
\end{proposition}

\begin{proof} 
  We   show   that   $G,B,C$    fulfil   the   the   assumptions   of
  Proposition~\ref{prop:duality} and Corollary~\ref{cor:duality}, from
  which the result will follow.

  The subgroup $B$ is cyclic and  characteristic in $G$, so $C$ is the
  only  subgroup  of  order  $r$  of  $G$,  and~\eqref{item:C-is-abelian}~and
  \eqref{item:C-is-char} of
  Proposition~\ref{prop:duality}  hold.  Let now $r =  q$; in this case $C =
  B$ and $G$ can  be of type 2, 3 or 11  and we always have
  $B   \cap   Z(G)   =  \{1\}$.    Moreover,   $\ord(\gamma(b))   \mid
  \ord_{(G,\circ)}(b)  \mid  q$   since  $(B,\circ) \le  (G,\circ)$  by
 Proposition~\ref{prop:2su3} and $\gamma\colon  (G,\circ)\to \Aut(G)$ is a
  morphism.    For    $G$   of    type  2,  3    or  11,
  $\Size{\Aut(G)/\Int(G)}$ is coprime to $q$, thus all the elements of
  order  $q$ in  $\Aut(G)$ belong  to  $\Int(G)$ and  so $\gamma(b)  =
  \iota(b^{-\sigma})$ for some $\sigma$.

  On the other hand, for $r = p$ we have $B = \Span{b}$, and $G$ is of
  type 4,  so that $p >  2$.  Here $Z(G)  = \Set{1}$, so that  we have
  only to show that for all $c  \in C = \Span{b^p}$ we have $\gamma(c)
  = \iota(c^{-\sigma})$.   According to Theorem~3.4 and  Remark~3.5 of
  ~\cite{classp2q},   the    Sylow   $p$-subgroup   of    $\Aut(G)   =
  \Hol(\cC_{p^2})$  is  a non-abelian  group  $X  = \cC_{p^2}  \rtimes
  \cC_p$ of  order $p^{3}$, spanned  by $\iota(b)$, of  order $p^{2}$,
  and  another element  $\psi$  of  order $p$  which  maps $b  \mapsto
  b^{1+p}$, and fixes  elementwise a Sylow $q$-subgroup $A$  of $G$ of
  one's choice.  The derived subgroup $\Span{\iota(b^{p})}$  of $X$ is
  central, of order $p$. We now quote the elementary
  \begin{remark}
    \label{rem:ah-the-p-powers}
    For $x, y \in X$ we have
    \begin{equation*}
      (x y)^{p} 
      = 
      x^{p} y^{p} [y, x]^{\binom{p}{2}}
      =
      x^{p} y^{p},
    \end{equation*}
    as $p$ is odd.
  \end{remark}
  We  have once  more $(B,\circ)  \le (G,\circ)$
  and   since   $p   >  2$,   by Corollary~\ref{cor:same_generator},
  $\ord_{(B,\circ)}(b^{\circ  k}) =  \ord_B(b^k)$  for  all $k$,  thus
  $\ord(\gamma(b^p)) \mid  p$.   If  $\gamma(b^p)  = 1$  we  are  done,
  otherwise    $\ord(\gamma(b^p))    =    p$.   Let    $\gamma(b)    =
  \iota(b^{-\sigma}) \psi^{t}$; then
  \begin{equation*}
    b^{\circ p} 
    = 
    b^{\sum_{i = 0}^{p-1}\gamma(b)^i} 
    = 
    b^{\sum_{i =0}^{p-1}(1+ipt)} 
    = 
    b^p,
  \end{equation*}
  so that                                    
  \begin{equation*}
    \gamma(b^p) 
    = 
    \gamma(b^{\circ p}) 
    = 
    \gamma(b)^p 
    =
    (\iota(b^{-\sigma})\psi^t)^p 
    = 
    \iota(b^{-\sigma  p}),
  \end{equation*}
  where we have used Remark~\ref{rem:ah-the-p-powers}.

  Finally, since $G$ is not  abelian, then $\Int(G)$ contains elements
  of order both $p$  and $q$, so there exists $a\in  G$ whose order is
  not   a  power   of  $r$   and  Proposition~\ref{prop:duality}   and
  Corollary~\ref{cor:duality} can be applied.
\end{proof}

\begin{theo}
  \label{teo:sylow}
  Let $G$ be a group of order $p^2q$ and $\gamma$ a GF on $G$.

  Then there exists a Sylow $p$-subgroup $A$ of $G$ which is
  $\gamma(A)$-invariant.

  In particular, for $p>2$, $G$ and $(G,\circ)$ have isomorphic Sylow
  $p$-subgroups.   
\end{theo}

\begin{proof}
  We show that there is always a  Sylow $p$-subgroup $A$ of $G$ which is
  $\gamma(A)$-invariant; then  the result  will follow  from
  Corollary~\ref{cor:sylow-gen},  
  since the groups of order $p^2$ are abelian. Clearly,
  this is always the case when  $A$ is characteristic; otherwise the set
  ${\mathcal P}$ of the Sylow $p$-subgroups of $G$ has $q$ elements. For
  each   $\gamma$,  the   group  $\gamma(G)$   acts  on   $\mathcal  P$,
  partitioning it  into orbits whose length  divides $\Size{\gamma(G)}$.
  So, denoting by $N_l$ the number of orbits of length $l$, we have
  $$\sum_{l \mid \Size{\gamma(G)}}N_ll = \Size{\mathcal{P}} = q \equiv 1
  \bmod p.$$  If $q \mid \Size{\ker(\gamma)}$,  then $\Size{\gamma(G)} =
  1,  p$  or  $p^2$  and   necessarily  $N_1\ge1$,  namely  there  exist
  $A\in\mathcal{P}$ which is $\gamma(G)$-invariant.

  This argument covers the cases when $G$ is of type 1, 4, 5, 6, 7, 8,
  9, 10 (that is, the Sylow $p$-subgroup is characteristic) and when $G$
  is of type 2, 3 or 11, and the Sylow $q$-subgroup $B$ is contained in
  $\ker (\gamma)$. So suppose $G$ of type 2, 3 or 11 and $B \not\le
  \ker (\gamma)$; here $B$ is characteristic and by
  Proposition~\ref{prop:duality_p2q} we get that
  $B\le\ker(\tilde\gamma)$. The 
  previous argument ensures that there exists a Sylow $p$-subgroup $A$
  of $G$ which is $\tilde\gamma(G)$-invariant and, by
  Proposition~\ref{prop:2su3}, it is also a Sylow $p$-subgroup of $(G, 
  \tilde\circ)$. 

  Now, $(G, \tilde\circ)$ is isomorphic to $(G,\circ)$ via the map $\inv \colon   x\mapsto  x^{-1}$   (see  the   proof  of
  Proposition~\ref{prop:gammatilde}), thus $A^{\inv}=A$  is also a Sylow $p$-subgroup of $(G,\circ)$. Using 
 Proposition~\ref{prop:2su3} again we get that $A$ is $\gamma(A)$-invariant.
 
  We can conclude  that, for each $G$ and for  each GF $\gamma$, there
  exists  a  Sylow   $p$-subgroup  of  $G$  which  is   also  a  Sylow
  $p$-subgroup  of  $(G,\circ)$.   Corollary~\ref{cor:sylow-gen}
  allows us to conclude  that $A$ and $(A,\circ)$  are isomorphic.
\end{proof}
We immediately get
\begin{corollary}
\label{cor:t33}
Let  $p>2$ and $q$ be distinct  primes. Let $\Gamma$ and $G$ be groups of order $p^2q$  with non isomorphic Sylow $p$-subgroups. Then $e'(\Gamma, G)=e(\Gamma,G)=0$.
\end{corollary}

\begin{remark}
  If $G$ is a group of order $p^2q$ with cyclic Sylow $p$-subgroups,
  then either $G$ has a unique Sylow $q$-subgroup or it is of type 4. In
  the latter case it follows from Subsection~\ref{subsec:G4} that there
  exists a unique Sylow $q$-subgroup $Q$ invariant under
  $\gamma(Q)$. So, a posteriori, Theorem~\ref{teo:sylow} is also true
  replacing Sylow $p$-subgroups with Sylow $q$-subgroups, but we do not
  have a general argument to prove it. 
\end{remark}

\section{Proof of Theorems~\ref{number}~and \ref{number_sb}}
\label{sec:proof}

We are now ready to prove our main Theorem~\ref{number} We will mostly
rely on the general results established in Sections~2~and 3, appealing
occasionally to ad hoc arguments.

In  enumerating the  GF's $\gamma$  on  $G$, we  will usually  tacitly
ignore the case $\gamma(G) = \Set{1}$, that is, $\ker(\gamma) = G$, as
it  corresponds   to  the   (trivial)  case   of  the   right  regular
representation.

According  to Theorem  \ref{teo:sylow}, if  $G$  is a  group of  order
$p^2q$,  with  $p,  q$  distinct   primes  and  $p>2$,  then  $G$  and
$(G,\circ)$  have  isomorphic  Sylow   $p$-subgroups.  Thus  to  prove
Theorem~\ref{number} we only  need to consider groups  $G$ with cyclic
Sylow $p$-subgroups, that is, those in Table~\ref{table:grp_aut}.  For
these groups each  type corresponds to an isomorphism  class.  We will
proceed by analysing the types/isomorphism classes under consideration
one by one.

We fix the following notation: $q$  and $p>2$ are distinct primes, $G$
is  a group  of  order  $p^2q$ with  cyclic  Sylow $p$-subgroups,  and
$\gamma\colon G \to \Aut(G)$ is a GF on $G$.

\begin{remark} 
  \label{Argument} 
  We discuss here a recurring  pattern which occurs in the application
  of Proposition~\ref{prop:1.4}.  Let $\Set{r,s}  = \Set{p,q}$ and let
  $A$ be  a Sylow  $r$-subgroup, and  $B$ be  a Sylow  $s$-subgroup of
  $G$. Clearly $G = A B$ and we  know that at least one of $A$ and $B$
  is characteristic.   In the following  we restrict our  attention to
  the  GF's   $\gamma  \colon  G   \to  \Aut(G)$  such  that   $B  \le
  \ker(\gamma)$.

  Suppose first $A$  is characteristic.  Then $\gamma$  is the lifting
  of $\gamma' = \gamma_{\restriction A}\colon  A \to \Aut(G)$.  On the
  other hand,  by Proposition~\ref{prop:1.4},  in this case  the RGF's
  $\gamma'  \colon A  \to  \Aut(G)$ which  can be  lifted  to $G$  are
  exactly those  for which  $B$ is  invariant under  $\Set{ \gamma'(a)
    \iota(a) \colon a \in A }$.
  
  If $A$ is not characteristic in  $G$, and thus $B$ is, the situation
  is slightly more involved.

  Consider the action  of $\gamma(G)$ on the set  $\mathcal{R}$ of the
  Sylow  $r$-subgroups of  $G$.  Since  by assumption  $\gamma(G)$ has
  order a  power of $r$,  Sylow's theorems imply that  $\gamma(G)$ has
  $N_{1}  >  0$  fixed  points  in  this  action.   Let  $\bar{A}  \in
  \mathcal{R}$ be  one of these  Sylow $r$-subgroups of  $G$ invariant
  under  $\gamma(G)$.   Then   $\gamma_{\restriction  \bar{A}}  \colon
  \bar{A} \to  \Aut(G)$ is  a RGF.   On the other  hand, since  $B$ is
  characteristic,  it is  invariant under  $\Set{ \gamma'(a)  \iota(a)
    \colon a \in  A } \le \Aut(G)$, and thus  each RGF $\gamma' \colon
  \bar{A} \to \Aut(G)$ can be lifted to  a GF on $G$.  It follows that
  when  $A$  is   not  characteristic,  each  $\gamma$   with  $B  \le
  \ker(\gamma)$ can be obtained as a lifting of a $\gamma'$ defined on
  a Sylow $r$-subgroup in $N_1$  ways, one for each Sylow $r$-subgroup
  $\bar{A}$ which is invariant under $\gamma(G)$.
\end{remark}

\subsection{G of type 1}
\label{subsec:G1}

In this case  $G= \cC_{p^2}\times \cC_q$; the Sylow  $p$-subgroup $A =
\Span{a}$ and  the Sylow $q$-subgroup  $B = \Span{b}$ are  both cyclic
and characteristic.   By Proposition~\ref{prop:2su3}  $A$ and  $B$ are
also subgroups of $(G, \circ)$,  for each operation $\circ$ induced on
$G$  by  $\gamma$.   Moreover,  $  \Aut(G)=\Aut(A)\times\Aut(B)  \cong
\cC_{p(p-1)}\times\cC_{q-1} $ is abelian.

\subsubsection{The case $B \le \ker(\gamma)$} 
\label{subs:B-in-ker}

This occurs  in particular  when $q \nmid  p - 1$,  that is,  $q \nmid
\Size{\Aut(G)}$.

Since    $A$   and    $B$    are   both    characteristic   in    $G$,
Remark~\ref{Argument} ensures that  the GF's on $G$  are in one-to-one
correspondence with the RGF's
\begin{equation*}
  \gamma' \colon A \to \Aut(G) .
\end{equation*}
On     the    other     hand,    since     $A$    is     cyclic,    by
Proposition~\ref{prop:images_gamma} each $\gamma'$ is uniquely defined
by    assigning    the    image    of    the    generator    $a$    as
$\gamma'(a)=\eta\in\Aut(G)$ where  $\ord(\eta) \mid  p^2$ and  $\eta =
(\eta_{\restriction A},\eta_{\restriction B})$.

For such  a $\gamma'$, the unique  $\gamma$ induced on $G$  defines an
operation       $\circ$       for      which,       according       to
Lemma~\ref{lemma:inverse_and_conjugacy},
\begin{equation}
  \label{eq:must-be-type-10}
  a^{\ominus  1}  \circ  b  \circ  a  =  a^{-\gamma(a)^{-1}  \gamma(b)
    \gamma(a)} b^{\gamma(a)} a = b^\eta.
\end{equation}

Thus if  $\eta_{\restriction B}  = 1$ (which  is the  only possibility
when $p \nmid q-  1$), then $(G, \circ)$ is abelian,  that is, of type
1:  there  are  $p$  choices  of  $\eta  \in  \Aut(G)$,  that  is,  of
$\eta_{\restriction A} \in  \Aut(A)$ of order dividing  $p$, with this
property.

As to the conjugacy classes, let
\begin{equation*}
  \gamma(a) :
  \begin{cases}
    a \mapsto a^{1 + p h}\\ b \mapsto b
  \end{cases}
\end{equation*}
and    consider    an    automorphism     of    $G$,    defined    for
$\gcd(u,p)=\gcd(v,q)=1$,
\begin{equation}
  \label{eq:beta4}
  \beta :
  \begin{cases}
    a \mapsto a^{u}\\ b \mapsto b^{v}.
  \end{cases}
\end{equation}
Then
\begin{multline*}
  \gamma^{\beta}(a)  = \gamma(a^{\beta^{-1}})  = \gamma(a^{u^{-1}})  =
  \gamma(a^{\circ \T(u^{-1})}) =\\ = \gamma(a)^{\T(u^{-1})} :
  \begin{cases}
    a \mapsto a^{1 + p h \T(u^{-1})}\\ b \mapsto b
  \end{cases}
\end{multline*}
Here   $\T$    is   the    inverse   of    the   function    $\s$   of
Lemma~\ref{lemma:arithmetic2}.  So  if $h =  0$ we have  the conjugacy
class of length $1$ of $\rho(G)$, whereas if $h \ne 0$, the stabiliser
is  given by  $\T(u^{-1})  \equiv  1 \pmod{p}$  and  any  $v$, so  the
stabiliser has order  $p (q - 1)$,  and there is a  conjugacy class of
length $p -1$.

If $p  \mid q-1$,  we can also  choose $\ord(\eta_{\restriction  B}) =
p$. There  are $p -  1$ choices  for such an  $\eta_{\restriction B}$,
which  paired with  the  $p$ choices  for  $\eta_{\restriction A}  \in
\Aut(A)$ of order  dividing $p$ yield $p (p-1)$ choices  for $\eta \in
\Aut(G)$.    In   this  case~\eqref{eq:must-be-type-10}   shows   that
$(G,\circ)$ is of type 2.

As to the conjugacy classes, if
\begin{equation}
  \label{eq:an-auto}
  \gamma(a) :
  \begin{cases}
    a \mapsto a^{1 + p h}\\ b \mapsto b^{r}
  \end{cases},
\end{equation}
with $r$ of order $p$, then
\begin{equation*}
  \gamma^{\beta}(a) =
  \begin{cases}
    a \mapsto a^{1 + p h \T(u^{-1})}\\ b \mapsto b^{r^{\T(u^{-1})}}
  \end{cases},
\end{equation*}
so the stabiliser is the same as  in the previous case, and we get $p$
classes of length $p-1$.

If $p^{2} \mid q-1$, we can also choose $\ord(\eta_{\restriction B}) =
p^2$.  As above, in this case there are $p^{2} (p-1)$ choices of $\eta
\in \Aut(G)$  with this property,  and~\eqref{eq:must-be-type-10} show
that $(G,\circ)$ is of type 3.

As to the  conjugacy classes, this time  $r$ in~\eqref{eq:an-auto} has
period $p^{2}$,  so for  the stabiliser we  need $\T(u^{-1})  \equiv 1
\pmod{p^{2}}$, that is, $u = 1$. Therefore the stabiliser has order $q
- 1$, and we get $p$ classes of length $p (p-1)$

\subsubsection{The case $B\not\le \ker(\gamma)$} 

Here $q  \mid p-1$, so that  $(G,\circ)$ can only  be of type 1  or 4.
Moreover,    $p   \mmid    \Size{\Aut(G)}$,    so    that   $p    \mid
\Size{\ker(\gamma)}$.

If $A \le \ker(\gamma)$, since $B$ is the unique Sylow $q$-subgroup of
$G$,  Remark~\ref{Argument}  yields  that  the  GF's  on  $G$  are  in
one-to-one  correspondence  with  the  RGF's  $\gamma'  \colon  B  \to
\Aut(G)$.  In  turn,  the  latter   are  uniquely  determined  by  the
assignment $b  \mapsto \gamma'(b)$,  where $\ord(\gamma'(b))  \mid q$.
Note that all  such $\gamma'$ are morphisms: this  follows either from
Corollary~\ref{cor:ordp}, or from Lemma~\ref{Lemma:gamma_morfismi}, as
$\gamma(G)$, of order $q$, acts trivially  on the group $G/A$ of order
$q$, so that  $[G, \gamma(G)] \le A \le \ker(\gamma)$.   For each such
$\gamma'$, the  unique $\gamma$  induced on  $G$ defines  an operation
$\circ$ such that:
\begin{equation*}
  b^{\ominus 1} \circ a\circ b = b^{-1}a^{\gamma(b)}b = a^{\gamma(b)}.
\end{equation*}

Since $B\not\le \ker(\gamma)$, then $\ord(\gamma'(b)) = q$.  There are
$q-1$   choices   for  such   a   $\gamma'(b)   \in  \Aut(G)$).   Here
$a^{\gamma(b)} \ne a$, and thus $(G, \circ)$ is of type 4.

As to the conjugacy classes, here
\begin{equation*}
  \gamma(b) :
  \begin{cases}
    a \mapsto a^{t}\\ b \mapsto b
  \end{cases}
\end{equation*}
with $t$ of order $q$ modulo $p^2$. With $\beta$ as in~\eqref{eq:beta4},
we have
\begin{equation*}
  \gamma^{\beta}(b)  = \gamma(b^{\beta^{-1}})  = \gamma(b^{v^{-1}})  =
  \gamma(b)^{v^{-1}} :
  \begin{cases}
    a \mapsto a^{t^{v^{-1}}}\\ b \mapsto b
  \end{cases},
\end{equation*}
which coincides with $\gamma$ if $v = 1$. Therefore the stabiliser has
order $p (p -1)$,  and we get a single conjugacy class  of length $q -
1$.

If $A \nleq \ker(\gamma)$, we show that $\gamma(b) = 1$, contradicting
$B \not \le \ker (\gamma)$.  The group $B$ is characteristic, hence by
Proposition~\ref{prop:2su3} it is also  a subgroup of $(G,\circ)$, and
by Corollary~\ref{cor:same_generator} $\ord_{(G,\circ)}(b)=q$, so that
$\ord(\gamma(b)) \mid q$.  In  particular, $b^{\gamma(b)} = b$.  Since
$\Size{\ker(\gamma)} = p q$ or $q$ in this case, and a group of type 4
has no  normal subgroup of this  order $(G,\circ)$ is abelian.   So we
have
\begin{equation*}
  b = a^{\ominus 1} \circ b \circ a = a^{-\gamma(b)} b^{\gamma(a)} a =
  a^{-\gamma(b)} a b^{\gamma(a)},
\end{equation*}
which gives $a^{\gamma(b)} = a$, so that $\gamma(b) = 1$ as claimed.

We summarise, including the right regular representations.
\begin{prop}
  \label{prop:G1}
  Let $G$ be of type 1, that is cyclic.
  
  Then in $\Hol(G)$ there are:
  \begin{enumerate}
  \item  $p$  regular  subgroups  of  type 1,  which  split  in  one
    conjugacy class of length $1$, and one conjugacy class of length
    $p - 1$.
  \item if $p \mid q-1$,
    \begin{enumerate}
    \item $p  (p-1)$ regular subgroups of  type 2, which split  in $p$
      conjugacy classes of length $p - 1$;
    \item $p^2(p-1)$  further regular subgroups  of type 3,  if $p^{2}
      \mid q-1$, which split in $p$  conjugacy classes of length $p (p
      - 1)$.
    \end{enumerate}
  \item if $q \mid p-1$,
    \begin{enumerate}
    \item  $q-1$ regular  subgroups of  type  4, which  form a  single
      conjugacy class.
    \end{enumerate}
  \end{enumerate}
\end{prop}

\subsection{G of type 2}
\label{subsec:G2}

In this  case $p \mid q-1$,  and $G = \cC_q  \rtimes_p \cC_{p^2}.$ The
Sylow $q$-subgroup $B = \Span{ b  }$ is characteristic in $G$. Here an
element of order  $p^{2}$ of $G$ induces an automorphism  of order $p$
of $B$.

We have
\begin{equation*}
  \Aut(G) \cong \Hol(\mathcal{C}_q) \times \mathcal{C}_p.
\end{equation*}
According  to  Subsection~4.5  of~\cite{classp2q}, the  second  direct
factor is generated by the automorphism $\psi$ of $G$ which fixes $b$,
and maps  every element of order  $p^2$ to its $({1+p})$-th  power. It
follows that $\psi$ fixes every element  of the unique subgroup of $G$
of order $p q$.

Since    $G$     is    non-abelian,     in    counting     the    GF's
Proposition~\ref{prop:duality_p2q} allows us to consider only the case
$B \le \ker(\gamma)$, and then  double the number of regular subgroups
we find.

Setting aside as always the  case of the right regular representation,
$\gamma(G)$  will thus  be a  subgroup of  $\Aut(G)$ of  order $p$  or
$p^2$.

A Sylow $p$-subgroup of $\Aut(G)$ is abelian, isomorphic to the direct
product of one  of the Sylow $p$-subgroups of  $\Hol(\cC_q)$ (which is
cyclic  of  order  $p^e$  where   $p^e  \mmid  q-1$),  and  the  group
$\Span{\psi}$.   Moreover,  the  elements of  $\Hol(\cC_q)$  of  order
dividing $p$ are of the form  $\iota(x)$ where $x$ is a $p$-element of
$G$.

\subsubsection{The case $\Size{\gamma(G)} = p^{2}$} 

This   case   can   only   occur   when  $p^2   \mid   q-1$,   as   by
Theorem~\ref{thm:gamma-for-regular}\eqref{item:gamma-is-hom}  we  have
$\gamma(G) \cong  (G,\circ)/\ker(\gamma)$, and this is  a cyclic group
by Theorem~\ref{teo:sylow}.  Therefore $\gamma(G)$  is generated by an
element $(\eta, \psi^t)$, where  $\eta\in \Hol(\cC_q)$ has order $p^2$
and $0  \le t  < p$.   Since $\eta^p$ is  an element  of order  $p$ of
$\Hol(\cC_q)$, we have $\eta^{p} = \iota(a)$  for an element of $a \in
G$  of  order  $p^2$.   Since  every  Sylow  $p$-subgroup  of  $G$  is
self-normalising,  $A =  \Span{a}$ is  the only  $\gamma(G)$-invariant
Sylow $p$-subgroup.

Once  one  of  the  $q$  Sylow  $p$-subgroups  $A$  has  been  chosen,
Remark~\ref{Argument} tells  us that to count  the GF's on $G$  we can
count    the   RGF's    $\gamma'   \colon    A   \to    \Aut(G)$.   By
Proposition~\ref{prop:images_gamma}, these are as many as the possible
images
\begin{equation}
  \label{eq:eta}
  \gamma'(a) = (\eta, \psi^t),
\end{equation}
with $\ord(\eta) = p^{2}$,  $0 \le t < p$, such  that $A$ is invariant
under   $\gamma'(a)$.  Since   $\Span{\psi}$  fixes   all  the   Sylow
$p$-subgroups of $G$,  it follows as above that $A  = \Span{a}$, where
$\eta^{p} = \iota(a)$.

Therefore, once  $A$ is chosen,  we have $p(p-1)$ choices  for $\eta$,
and $p$ choices for  $t$. So we have $q p^{2} (p-1)$  GF's on $G$ with
$\gamma(G)$ of order $p^{2}$.

In this case $(G,\circ)$ is always of  type 3. In fact, if $b^{\eta} =
b^{j}$, then $j$ has order $p^2$ modulo $q$ and
\begin{equation*}
  a^{\ominus 1}  \circ b  \circ a  = a^{-1}  b^{\gamma(a)} a  = a^{-1}
  b^{\eta} a = a^{-1} b^j a = b^{j \iota(a)}.
\end{equation*}
Since $\iota(a)$ is  an automorphism of $B$ of  order $p$, conjugation
by $a$ in $(G, \circ)$ is an automorphism of $B$ of order $p^{2}$.

As to  the conjugacy  classes, $\iota(B)$ conjugates  transitively the
Sylow      $p$-subgroups      of       $G$,      so      that,      by
Lemma~\ref{lemma:conjugacy}\eqref{item:invariance-under-conjugacy},
all  classes  have   order  a  multiple  of  $q$.    Since  the  Sylow
$p$-subgroups of $\Aut(G)$ are abelian, we then have for the action of
$\psi$ on one of our $\gamma$

\begin{equation*}
  \gamma^{\psi}(a)   =   \psi^{-1}    \gamma(a^{\psi^{-1}})   \psi   =
  \gamma(a^{1-p}),
\end{equation*}
so that  all classes have also  order a multiple of  $p$.  Finally, if
$\theta$ is an element of order $q  - 1$ of $\Aut(G)$ which fixes $a$,
then  $\Span{\theta}$  is in  the  stabiliser  of each  $\gamma$.   It
follows that we have $p (p -1)$ classes of length $q p$ here.

\subsubsection{The case $\Size{\gamma(G)} = p$} 
Here $\ker(\gamma)$ is the unique subgroup  of $G$ of index $p$. Since
$\gamma(G)$  acts  trivially  on  $G /  \ker(\gamma)$,  we  have  $[G,
  \gamma(G)]     \le    \ker(\gamma)$,     and    thus     by    Lemma
~\ref{Lemma:gamma_morfismi} all the GF's are morphisms $G \to \Aut(G)$
here.

In the case when $\gamma(G) = \Span{\psi}$, each Sylow $p$-subgroup of
$G$ is $\gamma(G)$-invariant, thus every  RGF on such a Sylow subgroup
lifts  to the  same GF  on  $G$. If  $\Span{a}$  is any  of the  Sylow
$p$-subgroups, there are  $p - 1$ choices for $\gamma(a)$,  and such a
choice determine $\gamma$  uniquely.  It is immediate to  check that $
a^{\ominus 1}  \circ b \circ  a =a^{-1} b a$, so that  the corresponding
groups $(G,\circ)$ are all of type 2.

As  to  the   conjugacy  classes,  $\psi$  is   central  in  $\Aut(G)$
and~\cite[Remark 3.3]{classp2q} implies  that $\Aut(G)$ acts trivially
on  $G  / \ker(\gamma)$  so  that  for  $\beta  \in \Aut(G)$  we  have
$\gamma^{\beta}(a) =  \gamma(a^{\beta^{-1}}) = \gamma(a)$, and  we end
up with $p - 1$ conjugacy classes of length $1$.

If  $\gamma(G) \ne  \Span{\psi}$, as  above  there is  a unique  Sylow
$p$-subgroup $A = \Span{a}$ fixed by $\gamma(G)$, and
\begin{equation*}
  \gamma(a) =  (\iota(a)^{-s}, \psi^t),  \quad \text{for  some $0  < s
    <p$, $0 \le t < p$.}
\end{equation*}
Since there are $q$ choices for the Sylow $p$-subgroup $A$, this gives
a total of $q (p - 1)p$ GF's. For the operation $\circ$ we have
\begin{equation*}
  a^{\ominus  1}\circ  b\circ a  =  a^{-1}  b^{\gamma(a)} a  =  a^{-1}
  b^{\iota(a)^{-s}} a = b^{a^{-s+1}}.
\end{equation*}
If $s \equiv 1 \pmod p$, then $(G,\circ)$ is of type 1 for each of the
$q p$ choices of $A$ and $t$.

If $s \not\equiv 1 \pmod p$, $(G,\circ)$  is of type 2; here there are
$q$ choices for $A$, $p - 2$ choices for $s$, and $p$ choices for $t$.

As to the conjugacy classes, with the above arguments we see that they
have all length a multiple of  $q$, and that the subgroup $\Span{\psi,
  \theta}$ of  order $p (q  - 1)$ is in  the stabiliser, so  that each
class has indeed length $q$.

We summarise, including the right and left regular representations.
\begin{prop}
  \label{prop:G2}
  Let  $G$ be  a  group of  type  2,  that is,  $G  = \cC_q  \rtimes_p
  \cC_{p^2}$.

  Then in $\Hol(G)$ there are:
  \begin{enumerate}
  \item 
    $2  p q$  regular subgroups  of  type 1,  which split  into $2  p$
    conjugacy classes of length $q$;
  \item
    \label{item:to-extend}
    $2 q p (p  - 2) + 2 p$ of  type 2, which split into $2  p (p - 2)$
    conjugacy classes  of length $q$,  and $2 p$ conjugacy  classes of
    length $1$;
  \item
    $2 q p^{2} (p - 1)$ further regular subgroups of type 3, if $p^{2}
    \mid q - 1$,  which split into $2 p (p -  1)$ conjugacy classes of
    length $q p$.
  \end{enumerate}
\end{prop}

\subsection{G of type 3}
\label{subsec:G3} 

In  this  case $G  =  \cC_q  \rtimes_1  \cC_{p^{2}}$, with  $p^2  \mid
q-1$. The Sylow $q$-subgroup $B=\langle b\rangle$ is characteristic in
$G$, and an element of order $p^{2}$ of $G$ induces an automorphism of
order $p^{2}$  on $B$.  Here $\Aut(G)  \cong \Hol(\mathcal{C}_q)$, and
since $G$ has trivial centre we have $\Int(G)\cong G.$

Since
\begin{equation*}
  \Size{\frac{\Aut(G)}{\Int(G)}} = \frac{q-1}{p^2}
\end{equation*}
is  coprime to  $q$, and  the Sylow  $p$-subgroups of  $\Aut (G)$  are
cyclic, we get $\gamma(G)\le\Int(G)$.

By Proposition~\ref{prop:2su3},  $B$ is  also a Sylow  $q$-subgroup of
$(G,\circ)$, so that $\Size{\gamma(B)}$ divides $q$, and
\begin{equation*}
  \gamma(B) \le \iota(B) = \Set{ \iota(b^x) : 0 \le x < q }.
\end{equation*}

Proposition~\ref{prop:duality_p2q} now allows us  to consider only the
case  $B \le  \ker(\gamma)$, and  then  double the  number of  regular
subgroups we find.

Now  Theorem~1  of~\cite{Curran}  (as   recorded  in  Theorem~3.2  and
Remark~3.3 of~\cite{classp2q}) yields that $\Aut(G)$ acts trivially on
$G  /  B$,  so  that  $[G,  \gamma(G)] \le  [G,  \Aut(G)]  \le  B  \le
\ker(\gamma)$,  and then  by Lemma~\ref{Lemma:gamma_morfismi}  all the
GF's are morphisms $\gamma : G \to \Aut(G)$ in this case.

If $\gamma(G) \ne  \Set{1}$, we claim that there is  exactly one Sylow
$p$-subgroup of $G$ which  is $\gamma(G)$-invariant.  In fact, $\Size{
  \gamma(G) } =  p$ or $p^2$, and $\gamma(G)$ is  a cyclic subgroup of
$\Span{ \iota(a) }$ for some $a  \in G$ with $\ord(a) = p^{2}$.  Since
every Sylow $p$-subgroup of $G$ is self-normalising, $A = \Span{a}$ is
the only $\gamma(G)$-invariant Sylow $p$-subgroup.

Let $\gamma(a) = \iota(a^{-s})$, for some $0 < s < p^{2}$.

In $G$  we have  $a^{-1} b  a = b^{t}$,  for some  $t$ of  order $p^2$
modulo $q$.  We get
\begin{equation}      
  \label{eq:conjugating}
  a^{\ominus   1}   \circ  b   \circ   a   =  a^{-1}b^{\gamma(a)}a   =
  a^{-1} b^{\iota(a)^{-s}} a = a^{-(1-s)} b a^{1-s} = b^{t^{1-s}}.
\end{equation}
Equation~\eqref{eq:conjugating} yields the following.
\begin{itemize}
\item If $s \equiv 1 \pmod{p^{2}}$, then $(G,\circ)$ is of type 1, and
  for each of the $q$ Sylow  $p$-subgroups $A$ of $G$ there is exactly
  one GF with this property.
\item If $s \equiv 1 \pmod{p}$ but $s \not\equiv 1 \pmod{p^{2}}$, then
  conjugation by $a$ in $(G, \circ)$ has order $p$, so $(G, \circ)$ is
  of type  2. For each  of the $q$  Sylow $p$-subgroups $A$  there are
  $p-1$ such GF's, so that we get $q (p - 1)$ GF's in this case.
\item  If $s  \not  \equiv 1  \pmod{p}$, conjugation  by  $a$ in  $(G,
  \circ)$ has order $p^{2}$, hence $(G, \circ)$ is of type 3. For each
  of the $q$ Sylow $p$-subgroups $A$ of  $G$ there are $p^{2} - p - 1$
  such choices of $s$ with $0 < s  < p^{2}$, so that we get $q(p^{2} -
  p - 1)$ groups in this case.
\end{itemize}
As to  the conjugacy  classes, $\iota(B)$ conjugates  transitively the
Sylow      $p$-subgroups       of      $G$,      so       that      by
Lemma~\ref{lemma:conjugacy}.\eqref{item:invariance-under-conjugacy}
each conjugacy class for  $s \ne 0$ has length a  multiple of $q$. The
cyclic complement  of order $q -  1$ of $\iota(B)$ in  $\Aut(G)$ which
contains $\gamma(a) = \iota(a^{-s})$  centralises $a$ and $\gamma(a)$,
so that,  by Lemma~\ref{lemma:conjugacy}, it centralises  $\gamma$. It
follows that the conjugacy classes have length precisely $q$.

We summarise, including the right and left regular representations.
\begin{prop}
  \label{prop:G3}
  Let     $G$     be     a     group     of     type     3,     namely
  $G=\mathcal{C}_q\rtimes_1\mathcal{C}_{p^2}$.    Then  in   $\Hol(G)$
  there are:
  \begin{enumerate}
  \item  $2 q$  regular  subgroups of  type 1,  which  split into  $2$
    conjugacy classes of length $q$;
  \item $2 q (p - 1)$ regular subgroups of type 2, which split into $2
    (p - 1)$ conjugacy classes of length $q$;
  \item $2 (1 +  q (p^2 - p - 1))$ regular subgroups  of type 3, which
    split into $2$ conjugacy classes of length  $1$, and $2 (p^2 - p -
    1)$ conjugacy classes of length $q$.
  \end{enumerate}
\end{prop}

\subsection{G of type 4}
\label{subsec:G4}

Here $q  \mid p  - 1$,  and $G  = \cC_{p^2}  \rtimes \cC_q$,  where an
element  of   order  $q$  acts   non-trivially  on  the   unique  Sylow
$p$-subgroup $B = \Span{b}$.  We have $\Aut(G) \cong \Hol(\cC_{p^2})$,
and $\Inn(G) \cong G$, as $Z(G) =  1$.  The groups $(G, \circ)$ can be
of type 1 or 4.

As  discussed  in  the  proof  of  Proposition~\ref{prop:duality_p2q},
$\Aut(G)$  has a  unique Sylow  $p$-subgroup, which  is isomorphic  to
$\cC_{p^2} \rtimes \cC_p$ where the  normal factor is $\Span{ \iota(b)
}$.  For the second factor we  have $p$ choices. In fact, according to
Theorem~3.4 and Remark~3.5 of~\cite{classp2q}, for each of the $p^{2}$
Sylow $q$-subgroups $A$ we can choose a generator $\psi$ of the second
factor such that $\psi \colon b \mapsto b^{1+p}$, and $\psi$ restricts
to the identity on  $A$: we will make a convenient  choice of $A$, and
thus $\psi$, later.  Note that if $\psi$ is the  identity on the Sylow
$q$-subgroup $A =  \Span{a}$, then it is also the  identity on the $p$
Sylow $q$-subgroups $\Span{a b^{p i}}$, for $0 \le i < p$.

Since the Sylow $q$-subgroups of $\Aut(G)$ are cyclic, the elements of
$\Aut(G)$ of order  $q$ are inner automorphisms,  given by conjugation
by an element of $G$ of order $q$.

By  Proposition~\ref{prop:duality_p2q},  we  can restrict  ourselves  to
counting the GF's such that $p$ divides $\Size{\ker(\gamma)}$.

\subsubsection{The case $\Size{\ker(\gamma)} = p^2$}   Here
$\ker(\gamma) =  B$ and $\Size{ \gamma(G)  } = q$. Hence  $\gamma(G) =
\Span{ \iota(a)  }$ for  some $a  \in G$ of  order $q$,  so that  $A =
\Span{a}$  is the  unique $\gamma(G)$-invariant  subgroup. Since  $[G,
  \gamma(G)] \le B = \ker(\gamma)$,  each such $\gamma$ is a morphism.
For the operation $\circ$ we have
\begin{equation*}
  a^{\ominus  1}  \circ   b  \circ  a  =  a^{-1}   b^{\gamma(a)}  a  =
  b^{\gamma(a) \iota(a)} .
\end{equation*}
Once  we have  made one  of  the $p^{2}$  choices for  $A$, the  value
$\gamma(a) = \iota(a)^{-1}$ will give an abelian group $(G, \circ)$ of
type 1,  while all  other $q  - 2$ choices  for $\gamma(a)$  will give
groups of type 4.

As to  the conjugacy  classes, $\iota(B)$ conjugates  transitively the
Sylow $q$-subgroups,  so that  all classes have  length a  multiple of
$p^{2}$. The cyclic  complement of order $p (p -  1)$ of $\iota(B)$ in
$\Aut(G)$ which contains $\gamma(G) = \Span{\iota(a)}$ centralises $a$
and  $\gamma(a)$ so  that,  by  Lemma~\ref{lemma:conjugacy}, it  fixes
$\gamma$. Hence the conjugacy classes have length precisely $p^{2}$.

\subsubsection{The case $\Size{\ker(\gamma)} = p q$} 

This  case does  not occur.   In fact,  since $\ker(\gamma)  \norm (G,
\circ)$, we  have that $(G,  \circ)$ is abelian  here. Thus if  $a \in
\ker(\gamma)$ is an element of order $q$, we have
\begin{equation*}
  a =  b^{\ominus 1} \circ a  \circ b = b^{-  \gamma(b)^{-1} \gamma(a)
    \gamma(b) } a^{\gamma(b)} b = a^{\gamma(b) \iota(b)}.
\end{equation*}
Now $\gamma(b)$ is an element of order $p$ in $\Aut(G)$. Therefore, by
Remark~\ref{rem:ah-the-p-powers}, we have
\begin{equation*}
  (\gamma(b) \iota(b))^{p} = \iota(b^{p}).
\end{equation*}
This implies that $a$ and $b^{p}$ commute, a contradiction.

\subsubsection{The case $\Size{ \ker(\gamma) } = p$}
\label{subsec: ker=p}

Here $\ker(\gamma)  = B^{p}$, and  $\Size{\gamma(G)} = p  q$.  Clearly
$\gamma(B)$ has order $p$, as $B  \le (G, \circ)$. An element of order
$q$ of $\gamma(G)$  will thus be of the form  $\gamma(a)$, for some $a
\in G$, which will be of order  $q$, as all elements of $G$ outside of
$B$ have order $q$.

Therefore
\begin{equation*}
  \gamma(a) = \iota(a^{-l} b^{m})
\end{equation*}
for some $0  < l < q$ and  $m$, so that $a^{-l} b^{m}$  has also order
$q$. Now choose a $\psi$ as above that fixes $a^{-l} b^{m}$.

$\gamma(b)$ will be an element of  order $p$ of $\Aut(G)$, that is, an
element of $\Span{\iota(b^{p}), \psi}$, an elementary abelian group of
order $p^{2}$. Since  $\gamma(G)$ is a subgroup of  $\Aut(G)$ of order
$p q$, and $p > q$, we will have that $\Span{\gamma(b)}$ is normalised
by $\iota(a^{-l} b^{m})$, an element  of order $q$.  Now $\iota(a^{-l}
b^{m})$ centralises $\psi$ by the choice of the latter, and normalises
but  does  not  centralise  $\Span{\iota(b^{p})}$.   In  other  words,
$\iota(a^{-l} b^{m})$  has two distinct  eigenvalues in its  action on
$\Span{\iota(b^{p}), \psi}$,  so that there are  two possibilities for
$\gamma(b)$ to be normalised by $\iota(a^{-l} b^{m})$.

If   $\gamma(b)    =\iota(b^{p   s})    \ne   1$,   for    some   $s$,
Proposition~\ref{prop:duality} yields that $p s \equiv -1 \pmod{p^2}$,
a contradiction.

Therefore $\gamma(b) = \psi^{t}$ for some  $t \ne 0$.  It follows that
$\gamma(G) = \Span  {\psi, \iota(a^{-l} b^{m})}$ is  abelian, and thus
$(G, \circ)$  is of  type 1, as  a group  of type 4  does not  have an
abelian quotient of order $p q$.

Comparing $a \circ b^{p} = a b^{p}$ with
\begin{equation*}
  b^{p} \circ a =  b^{p \gamma(a)} a = b^{p \iota(a^{-l})}  a = a b^{p
    \iota(a^{-l+1})}
\end{equation*}
in the abelian group $(G, \circ)$, we get $l = 1$.

Now
\begin{equation*}
  (a^{-1})^{\psi}  =  (a^{-1}  b^{m}  b^{-m})^{\psi}  =  a^{-1}  b^{m}
  (b^{-m})^{\psi} = a^{-1} b^{m} b^{-m (1 + p)} = a^{-1} b^{- p m},
\end{equation*}
so that, taking the inverse, we get
\begin{equation*}
  \label{eq:a-to-the-psi}
  a^{\psi} = b^{p m} a.
\end{equation*}

Comparing
\begin{equation*}
  a  \circ  b  = a^{\gamma(b)}  b  =  b^{p  m  t}  a  b =  b^{p  m  t}
  b^{\iota(a^{-1})} a
\end{equation*}
with
\begin{equation*}
  b \circ a = b^{\gamma(a)} a = b^{\iota(a^{-1})} a,
\end{equation*}
we obtain $p \mid m$. Write $m = p n$ for some $n$. We have thus
\begin{equation*}
  \gamma(b^{-p n}  a) =  \gamma(a) =  \iota(a^{-1} b^{p  n} )  = \iota
  (b^{-p n} a) ^{-1},
\end{equation*}
so that all  the GF's with kernel  of order $p$ can  be constructed as
follows.

Choose first  one of the  $p^{2}$ Sylow $q$-subgroups $A  = \Span{a}$.
Then define $\psi$ as  the automorphism of $G$ that is  the power $1 +
p$ on $B$, and fixes $a$. Finally define $\gamma$ as
\begin{equation}\label{eq:gamma-pq} 
  \begin{cases}
    \gamma(b) = \psi^{t}\\ \gamma(a) = \iota(a^{-1}).
  \end{cases}
\end{equation} 
It  is immediate  to  see that  $\gamma(a^{i}  b^{j}) =  \iota(a^{-i})
\psi^{t  j}$ defines  indeed a  GF satisfying~\eqref{eq:beta4}.   Note
that~\eqref{eq:beta4} determines $\Span{a}$ uniquely as the only Sylow
$q$-subgroup $A$  of $G$ which  is $\gamma(A)$-invariant. In  fact, if
$A=\Span{ab^k}$   is    $\gamma(A)$-invariant,   we    have,   writing
$b^{\iota(a^{-1})}=b^\lambda$,
\begin{equation*}
  A \ni (a b^k)^{\gamma(a b^k)}=ab^{k \lambda(1 + p t k)},
\end{equation*}
so that the latter equals $ab^k$, and we have
\begin{equation*}
  k(\lambda(1 + p t k)-1)\equiv 0\pmod{p^2}.
\end{equation*}
Since   $\lambda$  has   order  $q$   modulo  $p^2$,   we  have   that
$\lambda\not\equiv  1\pmod{p}$,   so  that   $k\equiv0\pmod{p^2}$  and
$A=\Span{a}$.

Therefore the $p^2$  choices for $A$ and  the $p - 1$  choices for $t$
yield $p^{2} (p - 1)$ choices for $\gamma$.

As   to   the   conjugacy    classes,   take   $\gamma$   defined   as
in~\eqref{eq:gamma-pq}.                                             By
Lemma~\ref{lemma:conjugacy}\eqref{item:invariance-under-conjugacy},
$\Span{\iota(b)}$  acts  regularly  on  the $\gamma$'s,  so  that  all
conjugacy classes have order a multiple of $p^{2}$.

Consider now the group  $R$ of automorphisms of $G$, of  order $p (p -
1)$, of the form
\begin{equation*}
  \beta :
  \begin{cases}
    a \mapsto a\\ b \mapsto b^{r}
  \end {cases} .
\end{equation*}
We claim that the stabiliser in  $R$ of any $\gamma$ is $\Span{\psi}$,
of order  $p$.  It  follows that  all conjugacy  classes have  order a
multiple of  $p-1$, and thus  all conjugacy classes have  order $p^{2}
(p-1)$.

In fact  one sees immediately, using  the fact that $\psi  \in R$, and
that the latter is cyclic, that
\begin{equation*}
  \begin{cases}
    \gamma^{\beta}(b)   =   \psi^{t  r^{-1}}\\   \gamma^{\beta}(a)   =
    \iota(a^{-1})
  \end{cases}
\end{equation*}
Thus $\gamma^{\beta} =  \gamma$ if and only if $r  \equiv 1 \pmod{p}$,
as claimed.

We summarise, including the right and left regular representations.
\begin{prop}
  \label{prop:G4}
  Let  $G$  be a  group  of  type 4,  namely  $G  = \cC_{p^2}  \rtimes
  \cC_q$. Then in $\Hol(G)$ there are:
  \begin{enumerate}
  \item  $2 p^{3}$  regular groups  of type  1, which  split into  $2$
    conjugacy classes  of length $p^{2}$,  and $2$ conjugacy  classes of
    length $p^{2} (p - 1)$;
  \item $2(1 + p^{2}  (q - 2))$ regular groups of  type 4, which split
    into  $2$  conjugacy classes  of  length  $1$,  and  $2 (q  -  2)$
    conjugacy classes of length $p^{2}$.
  \end{enumerate}
\end{prop}

\subsection{Proof of Theorem~\ref{number}~and \ref{number_sb}}

From Theorem~\ref{th:byott96}  we have that,  for each pair  of finite
groups $\Gamma, G$ with $\Size{\Gamma} = \Size{G}$,
\begin{equation*}
  e(\Gamma, G) =  \frac{\Size{\Aut(\Gamma)}} {\Size{\Aut(G)}} e'(\Gamma,
  G).
\end{equation*}
By  assumption  the Sylow  $p$-subgroups  of  the group  $\Gamma$  are
cyclic.  If the Sylow $p$-subgroups of  the group $G$ are also cyclic,
then    the     values    of     $e'(\Gamma,    G)$     computed    in
Propositions~\ref{prop:G1},    \ref{prop:G2},    \ref{prop:G3},    and
\ref{prop:G4}, and the cardinalities  of the automorphism groups given
in Table~\ref{table:grp_aut} yield the values of $e(\Gamma,G)$.

Propositions~\ref{prop:G1},    \ref{prop:G2},    \ref{prop:G3},    and
\ref{prop:G4} also yield, for  $G=(G,\cdot)$, the numbers of conjugacy
classes of  regular subgroups  of $\Hol(G)$, that  is, the  numbers of
isomorphism classes of skew braces $(G,\cdot,\circ)$.

If  the Sylow  $p$-subgroups of  the group  $G$ are  not cyclic,  then
Corollary~\ref{cor:t33} yields $e'(\Gamma, G)=e(\Gamma, G)=0$.

\section{Induced structures} 
\label{sec:split}

The  goal  of this  section  is  to  record  an alternative  proof  of
Proposition~\ref{prop:G1}, using the methods of~\cite{CRV16}.

Let  $L/K$ be  a  finite  Galois field  extension  and  let $\Gamma  =
\Gal(L/K)$.  A  Hopf-Galois structure on $L/K$  is called \emph{split}
if it is of type $G$ with $G=G_1\times G_2$ in a non-trivial way.

On the other  hand, let $F$ be a field  with $K\subseteq F\subseteq L$
such   that   $\Gamma'=\Gal(L/F)$   has   a   normal   complement   in
$\Gamma$. Assume that  $F/K$ and $L/F$ have  Hopf-Galois structures of
type $G_1$  and $G_2$, respectively.   Then $G=G_1\times G_2$  gives a
Hopf-Galois structure on $L/K$~\cite[Theorem 3]{CRV16}.  A Hopf-Galois
structure on $L/K$ is called \emph{induced} if it is obtained as above
for some field $F$ with $K\subsetneq F \subsetneq L$.

\begin{remark} It is immediate to see that each induced  Hopf-Galois structure is split, but the
converse it  is not  true in general  (see~\cite[Section 3.2]{CRV16}).
However, in the  case of extensions of  degree $p^2q$  all \emph{abelian} (split) structures are  induced. 
In fact, in this case the Galois group $\Gamma$ is always a semidirect product of a Sylow $p$-subgroup and a Sylow $q$-subgroup. On the other hand, if $N$ is an abelian group inducing a Hopf-Galois structure on $L/K$, then its Sylow $p$ and $q$-subgroups are characteristic, so they are stable under the action of $\rho(\Gamma)$ and \cite[Theorem 9]{CRV16} guarantees that each Hopf-Galois structure of type $N$ is induced.

In the case when $\Gamma$ has cyclic Sylow $p$-subgroups the only split structure is the cyclic one, so all split structures are induced in this case.
\end{remark}  
  
We give here a short sketch
of how  the number of cyclic  structures can be obtained  by computing
induced structures in the case of  Galois extension of degree $p^2q$, $p>2$ and cyclic Sylow $p$-subgroups.


So, let $L/K$  be a Galois extension of degree  $p^2q$ and assume that
$\Gamma$ has cyclic Sylow $p$-subgroups.  From Theorem~\ref{number} we
have that the number of cyclic  structures on $L/K$ is $p$, $pq$, $pq$
or $p^2$ for $\Gamma$ of type 1, 2, 3, or 4, respectively.

If  $\Gamma$ is  of  type 1,  2,  or 3,  then it  has  a normal  Sylow
$q$-subgroup  $B$. Denote  by  $A$ a  Sylow  $p$-subgroup of  $\Gamma$
(there is a  unique choice for $A$ if  $G$ is of type 1  and there are
$q$ choices if $G$  is of type 2 or 3) and let  $F$ be the fixed field
by $A$. Then $L/F$ is a Galois extension of degree $p^2$, so it admits
$p$ Hopf-Galois  structures, all  of cyclic  type (see~\cite[Corollary
  p.~3226]{Byo96}).  Moreover,  $B$ is  a normal  complement of  $A$ in
$\Gamma$,  so the  degree $q$  extension $F/K$  is almost  classically
Galois   and  it   always  admits   a  unique   Hopf-Galois  structure
(see~\cite[Theorem 5.2]{Par90}).   By~\cite[Theorem 3]{CRV16}  each of
these structures induces a split structure  on $L/K$, so we obtain the
$p$ or $pq$  structures, depending on $\Gamma$ being of  type 1, or of
type 2 or 3, already found in Theorem~\ref{number}.  If $\Gamma$ is of
type 2 or  3 those just described are the  only possible splittings of
$L/F$ which give induced structures.

If $\Gamma$  is of type 1,  in principle, we should  also consider the
tower $K\subseteq  M \subseteq  L$ where  $M$ is  the subfield  of $L$
fixed by the Sylow $q$-subgroup $B$.  However, in this case, since the
group $\Gamma$  splits, the  Hopf-Galois structures  on $L/M$  are the
same as  those on $F/K$ and  those on $M/K$  are the same as  those on
$L/F$, hence all of them have already been considered.

If $\Gamma$  is of type  4, the Sylow  $p$-subgroup is normal,  and we
have $p^{2}$ possible intermediate extensions $F$ with $L/F$ cyclic of
order  $q$  (one Hopf-Galois  structure  for  each of  them),  whereas
by~\cite[Theorem 9]{CS2018} $F/K$ has  a unique Hopf-Galois structure,
so we recover the $p^2$ cyclic structures as in Theorem~\ref{number}.


\providecommand{\bysame}{\leavevmode\hbox to3em{\hrulefill}\thinspace}
\providecommand{\MR}{\relax\ifhmode\unskip\space\fi MR }
\providecommand{\MRhref}[2]{%
  \href{http://www.ams.org/mathscinet-getitem?mr=#1}{#2}
}
\providecommand{\href}[2]{#2}

\end{document}